\documentclass[11pt]{amsart}

\RequirePackage{hyperref}
\hypersetup{pdfpagemode={UseOutlines},
bookmarksopen=true,
bookmarksopenlevel=0,
hypertexnames=false,
colorlinks=true, 
citecolor=teal, 
linkcolor=blue, 
urlcolor=magenta, 
pdfstartview={FitV},
unicode,
breaklinks=true,
}

\usepackage[margin=1.1in]{geometry}

\usepackage{epsfig,amscd,amssymb,amsmath,amsfonts}
\usepackage{amssymb}
\usepackage{tikz-cd}
\usepackage{epsfig}
\usepackage{verbatim}
\newtheorem{theorem}{Theorem}[section]
\newtheorem{theoremintro}{Theorem}[section]
\newtheorem{lemma}[theorem]{Lemma}
\newtheorem{proposition}{Proposition}[section]

\newtheorem*{theorem*}{Theorem}

\theoremstyle{definition}
\newtheorem{definition}[theorem]{Definition}
\newtheorem{corollary}[theorem]{Corollary}
\newtheorem{conjecture}[theoremintro]{Conjecture}

\newtheorem*{QIconjecture*}{QI Rigidity Conjecture}
\newtheorem*{PSconjecture*}{Pointed sphere Conjecture}

\theoremstyle{remark}
\newtheorem{remark}[theorem]{Remark}
\newtheorem{question}[theorem]{Question}

\numberwithin{equation}{section}

\usepackage{enumerate}

\usepackage{pgf,tikz}
\usetikzlibrary{arrows}
\usepackage{tikz-cd}

\usepackage{tikz-layers}
\usetikzlibrary{positioning}

\usetikzlibrary{patterns.meta}

\tikzstyle{stuff_fill}=[rectangle,fill=white,minimum size=1em]

\tikzstyle{stuff_fillc}=[circle,fill=white,minimum size=1em]

\tikzstyle{stuff_fillg}=[rectangle,fill=black!10,minimum size=1em]

\tikzstyle{stuff_nofill}=[rectangle,draw]

\usetikzlibrary{lindenmayersystems}
\pgfdeclarelindenmayersystem{cayley}{
  \rule{F -> F [ R [F] [+F] [-F] ]}
  \symbol{R}{
    \pgflsystemstep=0.5\pgflsystemstep
  } 
}

\def \R {\mathbf{R}}

\title[Groups with roughly similar left-invariant metrics]{Rough similarity of left-invariant Riemannian metrics on some Lie groups}

\author{Enrico Le Donne}

\address{Enrico Le Donne \\Department of Mathematics, University of Fribourg, Chemin du Musée~23, 1700 Fribourg, Switzerland \& 
University of Jyv\"askyl\"a, Department of Mathematics and Statistics, P.O. Box (MaD), FI-40014, Finland}
\email{enrico.ledonne@unifr.ch}

\author{Gabriel Pallier}
\address{Gabriel Pallier, Sorbonne Universit\'e, IMJ-PRG, 75252 Paris Cedex 05, France.}
\email{gabriel@pallier.org}

\author{Xiangdong Xie}
\address{Xiangdong Xie, Dept. of Mathematics  and   Statistics,   Bowling Green  State  University, 
  Bowling Green,  OH,   U.S.A.}
\email{xiex@bgsu.edu}

\date{\today}

\thanks{All three authors were partially supported by the European
Research Council (ERC Starting Grant 713998 GeoMeG `\emph{Geometry of
Metric Groups}'). Moreover,
the great part of this work was done while the three authors were
enjoying the stimulating environment of the University of Pisa.
E.L.D. was also partially supported by the Academy of Finland (grant
288501 `\emph{Geometry of subRiemannian groups}' and by grant 322898
`\emph{Sub-Riemannian Geometry via Metric-geometry and Lie-group
Theory}') and by the
Swiss National Science Foundation (grant 200021-204501
`\emph{Regularity of sub-Riemannian geodesics and applications}').
 X.X.   was partially  supported by
    Simons Foundation grant \#315130.  
}


\begin{document}

\maketitle


\begin{abstract}
We consider Lie groups that are either Heintze groups or  Sol-type groups, which generalize the three-dimensional Lie group SOL.
We prove that all  left-invariant Riemannian metrics on each such a Lie group are roughly similar via the identity.
This allows us to reformulate in a common framework former results by Le Donne-Xie, Eskin-Fisher-Whyte, Carrasco Piaggio, and recent results of Ferragut and Kleiner-M\"uller-Xie, on quasiisometries of these solvable groups. 
\end{abstract}

\tableofcontents


%

\section{Introduction}

\subsection{Main results}
In this paper, we compare  left-invariant Riemannian metrics on certain simply connected solvable Lie groups. The groups under study fall within two classes: 
\begin{itemize}
\item
Heintze groups, that is, simply connected solvable groups with Lie algebra $\mathfrak s$ such that $\mathfrak n = [\mathfrak s, \mathfrak s]$ has codimension $1$ in $\mathfrak{s}$ and $\mathfrak s$ splits as $\mathfrak n \rtimes \R$, where $\R$ acts on $\mathfrak n$ via a  derivation $D$   whose eigenvalues have positive real parts.
\item 
Sol-type groups, that is, simply connected solvable groups with Lie algebra $\mathfrak g$ such that $\mathfrak n = [\mathfrak g, \mathfrak g]$ has codimension $1$ in $\mathfrak{g}$ and $\mathfrak g$ splits as $\mathfrak n \rtimes  \R$, where $\R$ acts on $\mathfrak n$ via a  derivation $D$   whose eigenvalues have nonzero real parts, not all of the same sign, and such that $[\mathfrak n^{>0}, \mathfrak n^{<0}]=0$, where $\mathfrak{n}^{<0}$ (resp. $\mathfrak{n}^{>0}$) is the sum of eigenspaces with negative (resp. positive) real part.

\end{itemize}
Some relevant properties of these groups will be recalled in \S~\ref{sec:preliminary}; the main common feature of the groups we consider is to be simply connected, solvable, and have one-dimensional first cohomology, though the latter do not constitute a characterization.

In order to state our main result, recall that if $\phi \colon X \to Y$ is assumed to be a quasiisometry between  metric spaces $X$ and $Y$, then for some $c \geqslant 0$ there are positive constants $\lambda_-$ and $\lambda_+$ such that $\lambda_- d(x,x') - c \leqslant d(\phi(x), \phi(x'))$ and $\lambda_+ d(x,x') + c \geqslant d(\phi(x), \phi(x'))$ for every $x,x' \in X$. 
Say that the quasiisometry $\phi$ is a {\em rough similarity} if one can further take $\lambda_- = \lambda_+$ and a {\em rough isometry} if one can take $\lambda_- = \lambda_+ = 1$ in the inequalities above.

\begin{theoremintro}\label{main-heintze}
Let $S$ be a Heintze group and let $g_1$ and $g_2$ be  left-invariant Riemannian metrics on $S$  with distance function $d_1$ and $d_2$, respectively. Then the identity map
 $\text{Id}: (S, d_1)\rightarrow (S, d_2)$ is a rough similarity.
\end{theoremintro}

\begin{theoremintro}\label{main-soltype}
Let $G$ be  a  Sol-type  group and $g_1$ and $g_2$ be   left-invariant   Riemannian metrics on $G$   with distance function $d_1$ and $d_2$, respectively. Then the identity map
 $\text{Id}: (G, d_1)\rightarrow (G, d_2)$ is a rough similarity.
\end{theoremintro}

{\color{black}

Theorems~\ref{main-heintze} and \ref{main-soltype} imply the folowing statement at no cost: if $\varphi$ is an automorphism of a Heintze or Sol-type group, then $\varphi$ is a rough similarity with respect to any left-invariant Riemannian metric.
(It is an elementary fact that the {\em inner} automorphisms of any group $G$ equipped with a left-invariant distance $d$ are rough isometries; however, group automorphisms are in general no more than quasiisometries assuming in addition that $G$ is compactly generated and $d$ is proper geodesic.)
} 

Using Theorems~\ref{main-heintze} and \ref{main-soltype} we are able to reformulate certain results that appeared separately in the litterature. 
In the statement below, a Heintze group is of {\em special type} if it is a closed co-compact subgroup of a rank-one simple Lie group; {\em Carnot type} is a subclass of Heintze groups in which the nilradical is a Carnot group, and the derivation $D$ is a Carnot derivation of this group. 
For the background on Carnot groups, see for example \cite{ledonne_primer}.
The real shadow construction will be recalled along with precise definitions in \S~\ref{subsec:heintze-definition}.

The substantial part of the following theorem is provided by the given references, while its formulation depends on the results above.

\begin{theoremintro}
\label{th:reformulation}
Let $G$ belong to the following list:
\begin{enumerate}
\item
The Lie group $\mathrm{SOL}$ \cite{EFW2}.
\item
\label{ledonne-xie}
Heintze group whose real shadow is of Carnot type with reducible first stratum \cite{LeDonneXie}.
\item 
\label{carrascopiaggio}
Heintze group whose real shadow is not of Carnot type \cite{CarrascoOrliczHeintze}.
\item 
\label{kleiner-mueller-xie}
\begin{enumerate}
\item \label{KMX21}
Heintze group whose real shadow is of Carnot type, which is different from the special-type subgroups in $\mathrm{SO}(n,1)$ or $\mathrm{SU}(n,1)$, and whose nilradical is nonrigid in the sense of Ottazzi-Warhurst \cite{KMX21}.
\item \label{KMX22}
The Carnot-type Heintze group over the subgroup of unipotent triangular real $n \times n$ matrices, $n\geqslant 4$ \cite{LDX22f}.
\end{enumerate}
\item 
Non-unimodular Sol-type group \cite{FerragutThesis}.
\end{enumerate}
Equip $G$ with any left-invariant Riemannian metric with associated distance $d$. 
If $\phi \colon G \to G$ is a quasiisometry, then $\phi$ is a rough isometry with respect to $d$.
\end{theoremintro}

 
Note that, in general, the notion of a rough isometry of a group does not make sense because it depends on the left-invariant distance one choses on the group. {\color{black} In view of Theorems~\ref{main-heintze} and \ref{main-soltype}, the conclusion of Theorem \ref{th:reformulation} may also be stated in the following way: given any pair of left-invariant  Riemannian distances $d_1$ and $d_2$, every quasiisometry $(G,d_1) \to (G,d_2)$ is a rough similarity, whose similarity constant only depends on the pair $(d_1,d_2)$.}

We point out that the rigidity property of quasiisometries expressed in Theorem \ref{th:reformulation} is weaker than the  rigidity of quasiisometries (which means every self quasiisometry of a certain metric space  is at a finite distance from an isometry).  Every map at a finite distance from a isometry is a rough isometry.  However, depending on the space there may exist rough isometries that are not at finite distance from any isometry, and this does actually happen for the left-invariant metrics on certain Heintze and Sol-type groups. 

We also note the following:

\begin{itemize}
\item 
Carrasco Piaggio has stated the conclusion in an equivalent form when $G$ is as in \eqref{carrascopiaggio} and additionally purely real \cite{CarrascoOrliczHeintze}. His result subsumes former ones, the first of which being by Xie and Shanmugalingam \cite{SX}, the second one by Xie in \cite{XieLargeScale}. 
\item 
Case \eqref{ledonne-xie} subsumes former work by Xie in \cite{XieReducible}. {\color{black} The groups of class (C) defined in \cite[14.1]{PansuCCqi} fall within this family (See Remark \ref{rem:C-is-reducible}), and the early \cite[Theorem 4]{PansuCCqi} implies Theorem~\ref{th:reformulation} for these: their quasiisometries are actually a bounded distance away from inner automorphisms.}
\item 
{\color{black}
Cases \eqref{ledonne-xie} and \eqref{KMX21} overlap, though none of them imply the other. The groups considered in \cite[\S 14.3]{PansuCCqi} belong to both classes. 
Case \eqref{KMX22} is not implied by \eqref{ledonne-xie} nor by \eqref{KMX21}.}
\item The statements in the references given are not uniform, so the degree of reformulation varies.
\end{itemize}


Bringing them together, the cases \eqref{ledonne-xie}, \eqref{carrascopiaggio} and \eqref{kleiner-mueller-xie} of Theorem~\ref{th:reformulation} support the following conjecture:

\begin{conjecture}
\label{main-conj}
Let $S$ be a Heintze group, {\color{black} which is not among the special-type subgroups of} $\operatorname{SO}(n,1)$ or $\operatorname{SU}(n,1)$ for any $n \geqslant 2$. 
Equip $S$ with any left-invariant Riemannian metric. Then every self-quasiisometry of $S$ is a rough isometry.
\end{conjecture}

We will discuss further the relations and differences of Theorem \ref{th:reformulation} and Conjecture~\ref{main-conj} with quasi-isometric rigidity in the case of Heintze groups in \S \ref{subsec:comparison}. Especially, we will see there that Conjecture~\ref{main-conj} would follow from conjectures already explicitely stated in \cite{KMX21} and \cite{Cornulier:qihlc}. Keeping in mind that every homogeneous space of negative curvature is a Heintze group with a left-invariant metric, Conjecture~\ref{main-conj} can be considered as a precise version of the feeling expressed in the four lines before \S 1 in \cite{PansuCCqi}.

\subsection{Some context}

\subsubsection{Spaces of left-invariant metrics and comments on Theorems \ref{main-heintze} and \ref{main-soltype}}
The space of left-invariant Riemannian metrics on a given Lie group has been widely studied by differential geometers; let us rather restrict our discussion to the results that put an emphasis on large-scale geometry rather than on Lie groups, for we believe that this comparison is more instructive.

For a finitely generated group $\Gamma$, Gromov introduced a metric space denoted by $WM_\Gamma$ whose points are word metrics and the distance is measured by the logarithm of $\lambda$, where $(1/\lambda,\lambda)$ is the optimal pair of multiplicative quasiisometry constant between them \cite{AsInv}. 
The definition of this space itself is not straightforward, as one may consider several variants, especially one could compare metrics only through the identity map (as we do here), or through automorphisms, or even through arbitrary maps\footnote{One should also decide if roughly isometric or roughly similar metrics are to be identified; however this is not a deep distinction.}. One may also include metrics that are not word metrics, especially geometric metrics, induced by the Riemannian metrics on universal covers when $\Gamma$ is the fundamental group of a compact manifold.
The resulting space is in some sort reminiscent of Teichm\"uller space, and actually contains it when $\Gamma$ is a surface group.



Recently one of the variants of this space of left-invariant metrics was studied by Oreg\'on-Reyes in the case of word hyperbolic groups \cite[Theorem 1.3]{OregonReyes}. Oreg\'on-Reyes notes the analogy with Teichm\"uller spaces and identifies metrics that are roughly similar through the identity. 
\begin{theorem*}[Oreg\'on-Reyes]
Let $\Gamma$ be a word-hyperbolic group. Consider the space $\mathcal{D}(\Gamma)$ of left-invariant metrics on $\Gamma$ that are quasiisometric to word metrics, modded out by the equivalence relation $d \sim d'$ if $d$ and $d'$ are roughly similar through the identity.
Equip $\mathcal D(\Gamma)$ with the metric 
\[ \rho(d,d') := \inf \{ \log \lambda : \exists \sigma >0, \exists c \geqslant 0,\, \frac{\sigma}{\lambda} d - c \leqslant d' \leqslant \sigma \lambda d + c \}, \qquad \forall  d , d'  \in \mathcal D(\Gamma)  . \]
Then $\mathcal D(\Gamma)$ is unbounded.
\end{theorem*}

All the Heintze groups being Gromov-hyperbolic, Oreg\'on-Reyes result is in sharp constrast with ours, which suggests that Theorems~\ref{main-heintze} and~\ref{main-soltype} may be special to non-finitely generated groups. Whether they are special to connected Lie group is currently unknown to us and we ask specific questions in this direction at the end of this paper.

{\color{black}

\subsubsection{Differences with other forms of rigidity}\label{subsec:comparison}
Some of the papers cited in Theorem~\ref{th:reformulation} were dedicated to proving quasiisometric rigidity, and they are known for this, so that it may be useful to point out the differences of the conclusion of Theorem~\ref{th:reformulation} with quasiisometric rigidity itself.
Namely, the following is expected:
\begin{QIconjecture*}
Let $\Gamma$ be a finitely generated group.
\begin{enumerate}
\item 
If $\Gamma$ is quasiisometric to a Heintze group $S$, then $S$ is of special type, and $\Gamma$ is virtually a lattice in the rank-one simple Lie group containing $S$ as a co-compact closed subgroup.
\item 
If $\Gamma$ is quasiisometric to a Sol-type group $G$, then $G$ is unimodular, and $\Gamma$ is virtually a lattice in a Lie group $\widehat G$ containing $G$ as a co-compact closed subgroup.
\end{enumerate}
\end{QIconjecture*}

\begin{figure}
    \begin{center}

\begin{tikzpicture}[line cap=round,line join=round,>=angle 45,x=0.75cm,y=0.5cm]
\clip(-3,-3.9) rectangle (18,12);

\fill [color=black, draw, dash pattern = on 4pt off 2pt,  fill opacity=0.1] (-2,0) rectangle (12,8);

\fill [color=black, draw, dash pattern = on 4pt off 2pt, fill opacity=0.1] (5,-1) rectangle (17,11);

\draw (-2,0) node[stuff_fillg, anchor=north west] {Conjecture~\ref{main-conj}};
\draw (0,-1.5) node[anchor=center] {$\Downarrow$};
\draw (-2,-2.5) node[draw, anchor=north west] {Theorem~\ref{th:reformulation} \eqref{ledonne-xie}, \eqref{carrascopiaggio}, \eqref{kleiner-mueller-xie}};

\draw (-2,8) node[anchor=north west] {Global quasiconformal};
\draw (-2,7.2) node[anchor=north west] {homeomorphisms of};
\draw (-2,6.4) node[anchor=north west] {Carnot groups other than };
\draw (-2,5.6) node[anchor=north west] {abelian and Heisenberg};
\draw (-2,4.8) node[anchor=north west] { groups are bilipschitz};
\draw (-2,4) node[anchor=north west] {\cite[Conjecture 1.13]{LDX20p}};

\draw (17,-1) node[stuff_fillg, anchor=north east] {QI rigidity for all Heintze groups};
\draw (8.5,2.3) node[anchor=center] {Pointed sphere};
\draw (8.5,1.5) node[anchor=center] {conjecture \cite[19.104]{Cornulier:qihlc}};

\draw (8.5,7) node[anchor=center] {Rigidity of QIs};
\draw (8.5,6) node[anchor=center] {for $\mathbb H^{n}_{\mathbb H}$ and $\mathbb H^{2}_{\mathbb O}$ \cite{PansuCCqi}};

\draw (5.2,4.2) rectangle (16.8,10.8);

\draw (5.2,10.8) node[anchor=north west] {QI rigidity of $\mathbb H^{n \geqslant 3}_{\mathbb R}$ \cite{TukiaQCM}};

\draw (5.2,9.4) node[anchor=north west] {QI rigidity of $\mathbb H^{n}_{\mathbb C}$ \cite{ChowQIch}};

\draw (16.8,9) node[anchor=north east] {QI rigidity of $\mathbb H^{2}_{\mathbb R}$};

\draw (16.8,8) node[anchor=north east] {\cite{TukiaConvGroups} +};
\draw (16.8,7) node[anchor=north east] {\cite{GabaiConFuchsAnnals} or \cite{CassonJungreis}};


\draw (16.7,4.3)  node[anchor=south east] {\em special type};

\end{tikzpicture}
\end{center}
    \caption{Relation to QI rigidity. In this conjectural picture, the parts that are proved are framed within continuous lines.}
    \label{fig:conj-pic}
\end{figure}
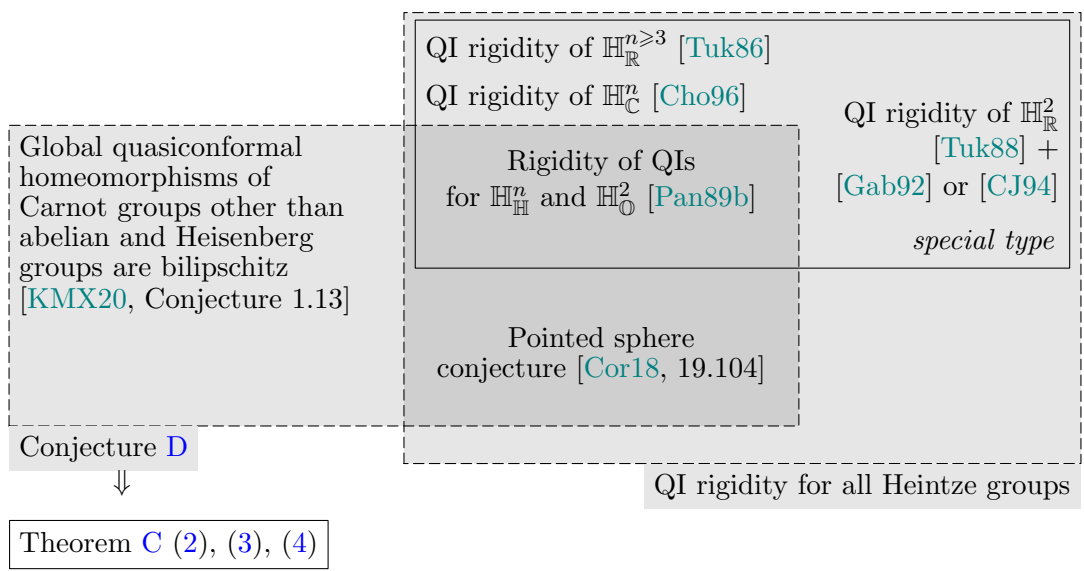

A common significant ingredient between quasiisometric rigidity and Theorem~\ref{th:reformulation} can be singled out in the case of Heintze groups that are not of special type. It is the following.

\begin{PSconjecture*}[Cornulier, {\cite[19.104]{Cornulier:qihlc}}]
Let $\phi$ be a self quasiisometry of a Heintze group, not of special type. Then the extension of $\phi$ to the Gromov boundary of $S$ fixes the unique boundary point that is fixed by all left-translations of $S$.
\end{PSconjecture*}

While conjectural in general, the following scheme of proof for Conjecture \ref{main-conj} should help the reader to understand our approach of some of the special cases of it in the present paper.

\begin{proof}[Proof of Conjecture~\ref{main-conj} assuming the Pointed Sphere Conjecture and {\cite[Conjecture 1.13]{LDX20p}}] (See Figure \ref{fig:conj-pic}).
Let $S=N \rtimes \mathbb R$ be a Heintze group as in the statement of Conjecture~\ref{main-conj}. Then,
\begin{itemize}
\item
Either $S$ is not of special type. In this case, the Gromov boundary of $S$ can be identified with a one-point compactification of $N$, with the boundary extension of $\phi$ stabilizing $N$. By \cite[Conjecture 1.13]{LDX20p}, then, the boundary extension of $\phi$ to $N$ equipped with a Carnot-Carath\'eodory metric should be bilipschitz, which, by \cite{SX}, implies that $\phi$ is a rough isometry.
\item 
Or $S$ is of special type. Then, by assumption, it is a closed cocompact subgroup of $\operatorname{Sp}(n,1)$ or $F_4^{(-20)}$ for some $n \geqslant 2$. The quasiisometries of $S$ are at a bounded distance from isometries of a left-invariant symmetric Riemannian metric on $S$ by \cite{PansuCCqi}, which implies by Theorem~\ref{main-heintze} that they are rough isometries of any left-invariant Riemannian metric, as mentionned in the paragraph below Theorem \ref{th:reformulation}. \qedhere
\end{itemize}
\end{proof}

The QI rigidity conjecture for Heintze groups, on the other hand, would follow from a combination of the QI rigidity for special-type groups, which were obtained in the 1980s and early 1990s (See Figure \ref{fig:conj-pic}), together with the fact that no finitely generated group should be quasiisometric to a non-special Heintze group.
We refer to \cite[Proof of Corollary 1.3]{SX} for how the Pointed Sphere Conjecture implies the last statement.

Finally, an analogy coming from the world of finitely-generated groups may lead one to think of quasiisometries of groups as large-scale counterparts of homotopy equivalences between compact manifolds. 
Following this analogy, at least in nonpositive curvature, rough isometries are large-scale counterparts to those homotopy equivalences that identify the marked length spectra (see e.g. \cite{Fujiwara}).
Theorem~\ref{th:reformulation} may then be considered analogous to the rigidity result that would consist in upgrading homotopy equivalence to marked length spectra isomorphism. Mostow's rigidity, which goes from homotopy equivalence to isometry, is strictly stronger, while length spectrum rigidity, which goes from the length spectrum to the isometry type, measures the difference.

}

\subsection{Organisation of the paper}

Section~\ref{sec:preliminary} collects preliminary material, namely definitions and three lemmas from Gromov-hyperbolic geometry.
Section~\ref{sec:heintze} proves Theorem~\ref{main-heintze} and Section~\ref{sec:soltype} proves Theorem~\ref{main-soltype}. Section~\ref{sec:soltype} is the technical heart of the paper, and Theorem~\ref{main-soltype} is significantly harder to prove than  Theorem~\ref{main-heintze}.
In Section~\ref{sec:reformulation} we start by proving a special case of Cornulier's Pointed Sphere Conjecture, which is instrumental in the reformulation of the main theorem of \cite{LeDonneXie}.
Next, we prove the other cases of Theorem~\ref{th:reformulation}.
In Section~\ref{sec:lamplighter} we point out that the conclusion of Theorem~\ref{main-soltype} does not hold in the Lamplighter group, and suggest a strengthening of the conclusion expressed by Theorems~\ref{main-heintze} and~\ref{main-soltype} which would be formulated in term of geometric actions that we did not reach in this paper.

\section{Preliminary}
\label{sec:preliminary}

\subsection{Notation}
If $G,H,N,S$ are Lie groups then $\mathfrak{g}, \mathfrak{h}, \mathfrak{n}, \mathfrak{s}$ are their Lie algebras.

\subsection{Gromov-hyperbolic geometry}

Let $T$ be a tree, $\xi\in \partial T$ a point in the ideal boundary, and $x, y\in T$. Then the intersection of the two rays $x\xi$, $y\xi$ is also a ray:  $x\xi\cap y\xi=z\xi$, where $x\xi$, $y\xi$ branch off at $z$.  The distance $d(x,y)$ equals the distance from $x$ to the branch point $z$ plus the distance from $y$ to the branch point $z$.   A similar statement holds for all  Gromov-hyperbolic spaces.  

The following lemma follows easily from the thin triangle condition.  We omit the proof. 

\begin{lemma}[See Figure~\ref{fig:distance-in-GH}]\label{distance in GH}
Let $X$ be  a proper geodesic  $\delta$-hyperbolic space, $\xi\in \partial X$, and $x,y\in X$.  Then there is a constant $C$ depending only on $\delta$, points $x'\in x\xi$, $y'\in y\xi$  such that  $d(x', y')\le C$ and the concatenation $xx'\cup x'y'\cup y'y$ is a 
   $(1, C)$-quasi-geodesic.  
     Here $x\xi$ denotes any geodesic joining $x$ and $\xi$; similarly for $y\xi$, $xx'$, $x'y'$, $y'y$.     
     In particular, 
    $|d(x,y)-(d(x,x')+d(y,y'))|\le C$. Furthermore,  $x', y'$ can be chosen so that they lie 
       on the same horosphere centered at $\xi$. 
\end{lemma}

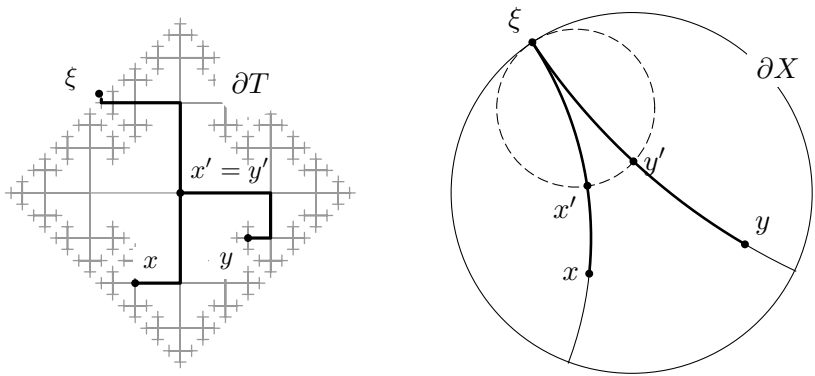
\begin{figure}
\begin{tikzpicture}[line cap=round,line join=round,>=triangle 45,x=0.6cm,y=0.6cm]
\clip(-15,-4.5) rectangle (5,4.5);
\draw(0,0) circle (4);
\draw [dash pattern = on 4pt off 2pt] (-1.24,1.88) circle (1.75);
\draw [shift={(9.79,11.24)}] plot[domain=3.72:4.27,variable=\t]({14.36*cos(\t r)},{14.36*sin(\t r)});
\draw [shift={(-8.78,-1)}] plot[domain=-0.36:0.58,variable=\t]({7.88*cos(\t r)},{7.88*sin(\t r)});
\draw (2.54,3.23) node[stuff_fill,anchor=north west] {$\partial X$};

\draw [shift={(9.79,11.24)}, line width = 1pt] plot[domain=3.72:4.18,variable=\t]({14.36*cos(\t r)},{14.36*sin(\t r)});
\draw [shift={(-8.78,-1)}, line width = 1pt] plot[domain=-0.1:0.58,variable=\t]({7.88*cos(\t r)},{7.88*sin(\t r)});

\fill [color=black] (-2.2,3.34) circle (1.5pt) node[above left] {$\xi$} ;
\fill [color=black] (-0.94,-1.79) circle (1.5pt) node[left] {$x$} ;;
\fill [color=black] (2.51,-1.14) circle (1.5pt)node[above right] {$y$} ;
\fill [color=black] (-0.98,0.16) circle (1.5pt)node[below left] {$x'$} ;
\fill [color=black] (0.04,0.7) circle (1.5pt) node[right] {$y'$} ;

\draw [shift={(-10,0)}, line width=0.5pt, color=black!40] l-system [l-system={cayley, axiom=[F] [+F] [-F] [++F], angle=90, step=1.2cm, order=4}];

\fill [shift={(-10,0)},color=black] (0,0) circle (1.5pt) node[above right] {\small $x'=y'$} ;
\fill [shift={(-10,0)},color=black] (-1,-2) circle (1.5pt)  ;
\draw [shift={(-10,0)},color=black] (-1.05,-1.9) node[stuff_fill, above right] {\small $x$};

\fill [shift={(-10,0)},color=black] (1.5,-1) circle (1.5pt)  ;
\draw [shift={(-10,0)},color=black] (1.4,-1.1) node[stuff_fill, below left] {\small $y$};

\fill [shift={(-10,0)},color=black] (-1.8,2.2) circle (1.5pt)  ;
\draw [shift={(-10,0)},color=black] (-2.05,2.05) node[stuff_fill, above left] {\small $\xi$};

\draw [shift={(-10,0)}, line width = 1.2pt] (-1,-2) -- (0,-2) -- (0,0) -- (2,0) -- (2,-1) -- (1.5,-1);

\draw [shift={(-10,0)}, line width = 1.2pt] (0,0) -- (0,2) -- (-1.75,2) -- (-1.75,2.08);

\draw (-9,3) node[stuff_fillc,anchor=north west] {$\partial T$};

\end{tikzpicture}
\caption{Lemma~\ref{distance in GH} in a tree and in the hyperbolic plane.}\label{fig:distance-in-GH}
\end{figure}

The next two lemmas are more involved, and will not be used before Section~\ref{sec:soltype} where they serve as a preparation for the key step of Theorem~\ref{main-soltype}.
The starting point is a  well-known fact about simply connected Riemannian manifolds with sectional curvature  bounded above by a negative constant: if $p, q$ lie on the same horosphere then the length of every path joining $p$ and $q$ outside the horoball is at least exponential in $d(p,q)$. 
  For  completeness, we provide a proof that is also true for Gromov-hyperbolic spaces.
  
  \begin{lemma}\label{comparison-horosphere}
   Let $X$ be  a proper geodesic   $\delta$-hyperbolic space, $\xi\in\partial X$, $S$ a horosphere centered at $\xi$, and    $B$  the horoball bounded by $S$. 
       Then for every $p, q\in S$ and every path 
    $c$ in $X\backslash B$ joining $p$ and $q$,  the length of $c$ satisfies  $\ell(c)\ge  2^{\frac{d(p,q)-C-2}{2\delta}} -C$, where   
      $C$ is a constant depending only on $\delta$. 
  
  \end{lemma}

  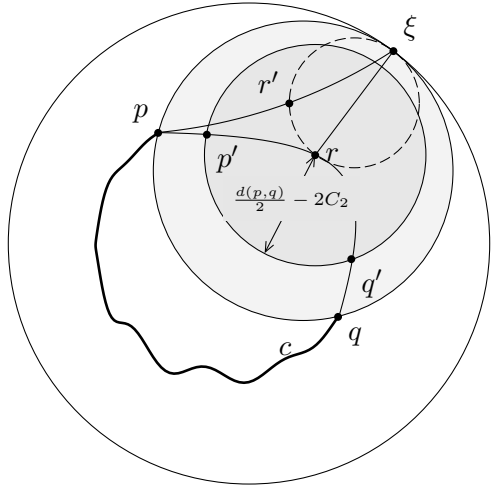
\begin{figure}
  
  \begin{tikzpicture}[line cap=round,line join=round,>=angle 45,x=0.8cm,y=0.8cm]
\clip(-4.1,-4.1) rectangle (4.1,4.1);
\draw [fill=black,fill opacity=0.05] (0.9,1.21) circle (2.49);
\draw(0,0) circle (4);
\draw [dash pattern = on 4pt off 2pt] (1.75,2.34) circle (1.08);
\draw [fill=black,fill opacity=0.05] (1.1,1.47) circle (1.84);

\draw [samples=50,domain=-0.54:0.46,rotate around={-144.8:(3.63,2.56)},xshift=2.90cm,yshift=2.05cm] plot ({2.53*(1+\x*\x)/(1-\x*\x)},{1.56*2*\x/(1-\x*\x)});
\draw [samples=50,domain=-0.99:0.99,rotate around={-144.8:(3.63,2.56)},xshift=2.90cm,yshift=2.05cm] plot ({2.53*(-1-\x*\x)/(1-\x*\x)},{1.56*(-2)*\x/(1-\x*\x)});

\draw [shift={(-2.3,10.39)}] plot[domain=4.8:5.29,variable=\t]({1*8.58*cos(\t r)+0*8.58*sin(\t r)},{0*8.58*cos(\t r)+1*8.58*sin(\t r)});
\fill [color=black] (2.4,3.2) circle (1.5pt) node[anchor=south west]{$\xi$};
\fill [color=black] (-1.51,1.84) circle (1.5pt)node[anchor= south east]{$p$};
\fill [color=black] (1.48,-1.22) circle (1.5pt)node[anchor=north west]{$q$};
\fill [color=black] (1.1,1.47) circle (1.5pt) node[anchor=west]{$r$};
\fill [color=black] (-0.95,5.71) circle (1.5pt);
\fill [color=black] (0.67,2.33) circle (1.5pt) node[anchor=south east]{$r'$};
\draw (1.1,1.47) -- (2.4,3.2);

\fill (-0.7,1.81) circle (1.5pt) node[anchor=north west]{$p'$};

\fill (1.7,-0.26) circle (1.5pt) node[anchor=north west]{$q'$};

\draw [shift={(-0.36,-0.03)}, line width = 1pt] plot[domain=2.12:5.71,samples=100,variable=\t]({2.19*cos(\t r)},{2.19*sin(\t r)+0.05*(\t-2.12)*(5.71-\t)*sin(10*\t r)});

\draw [<->] (1.1,1.47) -- (0.26,-0.16);
\draw (0.7,0.75) node[stuff_fillg, anchor=center]{\tiny $\frac{d(p,q)}{2}-2C_2$};

\draw (0.6,-1.75) node[anchor=center]{$c$};
\end{tikzpicture}
\caption{Proof of Lemma~\ref{comparison-horosphere}.}\label{fig:2.2}
  \end{figure}

  \begin{proof}
  Let $\gamma$ be a geodesic between $p$ and $q$ and $r$ be  a ``highest'' point on $\gamma$,  that is,     for any Busemann function $b$ centered at $\xi$, we have
   $b(r)=\min\{b(x)| x\in \gamma\}$, see Figure~\ref{fig:2.2}.   
     We   claim  $B(r, d(p,q)/2-2C_2)\subset B$ for some   constant $C_2$ depending only on $\delta$.     To see this,
       we   first  notice that  $d(r,p)\ge d(p,q)/2$ or $d(r,q)\ge d(p,q)/2$. Without loss of generality we assume 
   $d(r,p)\ge d(p,q)/2$.    
          Next we  consider the path  $\gamma[p, r]\cup r\xi$, where $\gamma[p,r]$ denotes the segment of $\gamma$ between $p$ and $r$.   Since $r$ is   a ``highest'' point on $\gamma$,  it is clear that  $\gamma[p, r]\cup r\xi$ is a $(1,C_1)$ quasi-geodesic  from $p$ to $\xi$  for some constant $C_1$ depending only on $\delta$.   By the Morse Lemma\footnote{Incidentally, the version of the current lemma where $\gamma$ avoids a ball rather than a horoball is a key ingredient in the proof of the Morse Lemma itself. So it actually occurs twice in this proof.}, the Hausdorff distance between $p\xi$ and $\gamma[p, r]\cup r\xi$  is bounded above by a constant $C_2$ depending only on $\delta$.
           Hence $d(r, x)\le C_2$ for some $x\in p\xi$.   Let   $r'\in p\xi$ be the point   at the same height as $r$, that is, $b(r')=b(r)$. Then $d(x, r')\le C_2$ (comparing the Busemann function of $x$ and $r'$ with respect to $\xi$)
           and so by the triangle inequality $d(r, r')\le 2C_2$. It follows that 
            $$b(r)-b(p)=d(r', p)\ge d(p, r)-d(r',r)\ge d(p,r)-2C_2\ge  d(p,q)/2-2C_2.$$ The claim follows from this.
   
   Let $p'\in \gamma$ between $p$ and $r$ such that $d(r, p')=d(p,q)/2-2C_2$, and $q'\in \gamma$ between $r$ and $q$ such that $d(r, q')=d(p,q)/2-2C_2$, see Figure~\ref{fig:2.2}.  
    Then the path 
   $c'=\gamma[p',p]\cup c\cup \gamma[q,q']$ joins $p'$ and $q'$ and lies outside the ball $B(r, d(p,q)/2-2C_2)$. 
      By Proposition 1.6 on  page 400 of \cite{BH}, 
     the length of   $c'$ satisfies 
     $$\ell(c')\ge    2^{\frac{d(p,q)/2-2C_2-1}{\delta}} .$$  
     The lemma follows  with $C=4C_2$   since $\ell(c)=\ell(c')-d(p,p')-d(q,q')$ and $d(p,p')+d(q,q')=d(p,q)-d(p', q')=4C_2$.

  \end{proof}

  For any subset $A\subset X$ as in the above lemma, let  
  $$H(A):=\sup\{b(x)|x\in A\}-\inf\{b(x)|x\in A\}$$
       be the height change of points in $A$.     Such a quantity can also be  similarly defined for  subsets of a Sol-type group since there is a notion of height in a Sol-type group. 
  
  \begin{lemma}\label{path-outside}
   Let $X$ be  a proper geodesic   $\delta$-hyperbolic space, $\xi\in\partial X$,   $b$ a Busemann function based at $\xi$,  $S$ a horosphere centered at $\xi$,  and 
      $B$ the  horoball  with boundary $S$. Let  $p, q\in S$ and 
    $c: [0,l]\rightarrow X\backslash B$ a path    with $c(0)=p$, $c(l)=q$. Then, \begin{enumerate}
    \item 
    The length of $c$ satisfies 
     $\ell(c)\ge  2 H(c)+ 2^{\frac{ d(p,q)-C-2}{2\delta}}-C-5d(p,q)$.  
    \item 
      Assume $H(c)>d(p,q)$. Then  there are   $0\le s<s'\le t'<t\le l$  such that $b(c(s))=b(c(t))<b(c(s'))=b(c(t')) $,    $d(p, q)<|b(c(s))-b(c(s'))|\le 2 d(p,q)$, and
       $\ell(c|_{[s, s']})+\ell(c_{[t', t]})\ge   2^{\frac{ d(p,q)-C-2}{2\delta}}-C-d(p,q).$
    \end{enumerate}

        {Here $C$ is the constant from Lemma~\ref{comparison-horosphere}, especially it only depends on $\delta$}.
  \end{lemma}

  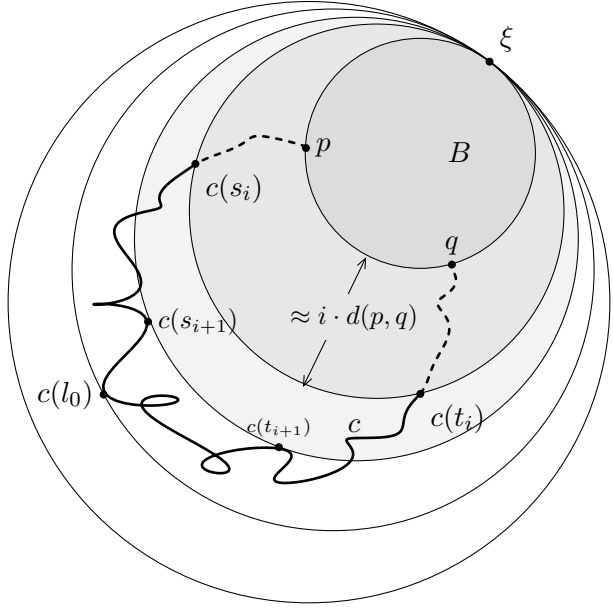
\begin{figure}
  
  \begin{tikzpicture}[line cap=round,line join=round,>=angle 45,x=1.0cm,y=1.0cm]
\clip(-4.1,-4.1) rectangle (4.1,4.1);
\draw [fill=black,fill opacity=0.05] (0.9,1.21) circle (2.49cm);
\draw(0,0) circle (4cm);

\fill [color=black] (2.4,3.2) circle (1.5pt) node[anchor=south west]{$\xi$};
\fill [color=black] (-1.51,1.84) circle (1.5pt)node[anchor= north west]{$c(s_i)$};
\fill [color=black] (1.48,-1.22) circle (1.5pt)node[anchor=north west]{$c(t_i)$};

\draw [color=black] (0.32,0.43) circle (3.47cm);

\draw [fill=black, fill opacity=0.05] (0.63,0.84) circle (2.95cm);

\draw [shift={(-0.36,-0.03)}, line width = 1pt] plot[domain=2.12:5.71,samples=200,variable=\t]({2.19*cos(\t r)+0.12*(\t-2.12)*(5.71-\t)*cos(15*\t r)},{2.19*sin(\t r)+0.09*(\t-2.12)*(5.71-\t)*sin(10*\t r)});

\draw [shift={(-0.36,-0.03)}, line width = 1pt, dash pattern=on 2pt off 3pt] plot[domain=1.48:2.12,samples=100,variable=\t]({2.19*cos(\t r)+0.05*(\t-2.12)*(5.71-\t)*cos(15*\t r)},{2.19*sin(\t r)+0.04*(\t-2.12)*(5.71-\t)*sin(10*\t r)});

\draw [shift={(-0.36,-0.03)}, line width = 1pt, dash pattern=on 2pt off 3pt] plot[domain=5.71:6.55,samples=100,variable=\t]({2.19*cos(\t r)+0.05*(\t-2.12)*(5.71-\t)*cos(15*\t r)},{2.19*sin(\t r)+0.04*(\t-2.12)*(5.71-\t)*sin(10*\t r)});

\draw [fill=black, fill opacity = 0.05] (1.48,1.98) circle (1.53cm);

\fill [color=black] (-2.14,-0.26) circle (1.5pt)node[anchor=west]{\small $c(s_{i+1})$};
\fill [color=black] (-0.40,-1.93) circle (1.5pt)node[anchor= south]{\tiny $c(t_{i+1})$};

\fill [color=black] (-2.73,-1.23) circle (1.5pt)node[anchor= east]{ $c(l_0)$};

\fill [color=black] (-0.04,2.05) circle (1.5pt)node[anchor=west]{ $p$};
\fill [color=black] (1.9,0.5) circle (1.5pt)node[anchor= south]{ $q$};


\draw (0.6,-1.65) node[anchor=center]{$c$};

\draw [<->] (0.76,0.6) -- (-0.08,-1.13);
\draw (0.6,-0.2) node[stuff_fillg, anchor=center]{\small $\approx i \cdot d(p,q)$};

\draw (2,2) node[anchor=center]{$B$};
\end{tikzpicture}
\caption{Proof of Lemma~\ref{path-outside} in the situation where $H(c)$ is much larger than $d(p,q)$.}
  \end{figure}

  \begin{proof}  The lemma follows immediately from Lemma~\ref{comparison-horosphere}
  when $H(c)\le d(p, q)$. So we assume $H(c)>d(p,q)$.  
  Let $b(p)=b_0<b_1<\cdots <b_m=b(p)+H(c)$ be such that $ d(p, q)<b_{i+1}-b_i\le 2 d(p,q)$.   
     Let $l_0\in [0, l]$ be such that $c(l_0)$ is a  lowest point on $c$, that is,  $b(c(l_0))=\max\{b(x)|x\in c\}$.
   For each $1\le i< m$, let $s_i\in [0, l_0]$ be the last $t$ in $[0, l_0]$  satisfying $b(c(t))=b_i$ and similarly   let $t_i\in [l_0, l]$   be  the first  $t$ 
     in $[l_0,l]$    
    satisfying $b(c(t))=b_i$. 
    We also set    $s_0=0$, $t_0=l$  and $s_m=t_m=l_0$.   
      The choices of
      $s_i$ and $t_i$ imply that $c|_{[s_i, s_{i+1}]}$  and $c|_{[t_{i+1}, t_{i}]}$
         lie  below the horosphere $b=b_i$, that is, $b(c(t))\ge b_i$ for $t\in [s_i, s_{i+1}]\cup [t_{i+1}, t_i]$.     Let $k$ be   the integer such that 
          $d(c(s_i), c(t_i))\ge  d(p, q)$ for all $i\le k$ and $d(c(s_{k+1}), c(t_{k+1}))<d(p,q)$.   Let $\gamma$ be a geodesic between $c(s_{k+1})$ and $c(t_{k+1})$. Then the path 
           $c|_{[s_k, s_{k+1}]}\cup \gamma\cup c|_{[t_{k+1}, t_k]}$ is a path below the horosphere $b=b_k$ joining $c(s_k)$ and $c(t_k)$.
            Now Lemma~\ref{comparison-horosphere}    
               implies 
               $$\ell(c|_{[s_k, s_{k+1}]})+\ell(c_{[t_{k+1}, t_k]})\ge   2^{\frac{ d(p,q)-C-2}{2\delta}}-C-d(p,q).$$  Now for each $i\not=k$  by considering the height change 
  we get
             $\ell(c|_{[s_i, s_{i+1}]}),   \ell(c|_{[t_{i+1}, t_{i}]}) \ge |b_i-b_{i+1}|$.     Now  (1)  follows   as $\ell(c)=\sum_i  (\ell(c|_{[s_i, s_{i+1}]})+ \ell(c|_{[t_{i+1}, t_{i}]}))$   and 
                $\sum_i |b_i-b_{i+1}|=H(c)$.    (2) holds with $s=s_k$, $s'=s_{k+1}$, $t'=t_{k+1}$, $t=t_k$.

  \end{proof}

\subsection{Heintze groups}
\label{subsec:heintze-definition}

Given a derivation $D$ on a Lie algebra $\mathfrak n$, we denote by  $\mathfrak n \rtimes_{D} \mathbb R$ the Lie algebra obtained as a semidirect product $\mathfrak n \rtimes \mathbb R$
 where $1\in \mathbb R$ acts on $\mathfrak n$ by the derivation $D$.

 \begin{definition}[Heintze group]\label{def:heintze-group}
Let $N$ be a nilpotent simply connected Lie group and let $D $ be a derivation of $\mathfrak n$ that has only eigenvalues with positive real parts and the smallest one has real part equal to one.
A {\em Heintze group} is a simply connected solvable Lie group having Lie algebra $\mathfrak n \rtimes_D \mathbb R$.
\end{definition}

Heintze groups are Gromov-hyperbolic. 
Even better, they have at least one lef-invariant Riemannian metric of negative sectional curvature \cite{Heintze}, and this is a characterization among connected Lie groups.

\begin{definition}[Carnot-type Heintze group]
A Heintze group $S$ is of {\em Carnot type} if $\ker(D-1)$ Lie generates $\mathfrak n$; this does not depend on the derivation $D$ such that $\mathfrak s \simeq \mathfrak n \rtimes_D \mathbb R$.
\end{definition}

The rank-one type Heintze groups defined in the Introduction are of Carnot type.

A Heintze group has a distinguished family of horospheres, disregarding the choice of a particular left-invariant Riemannian metric. 
Those are left cosets of the derived subgroup $N$.
By focal point of a Heintze group we mean the limit point of the subgroup $N=[S,S]$ in the Gromov boundary. 
When $S$ is naturally acting on its Gromov boundary, this point is the only one fixed by $S$.

\begin{definition}[Real shadow]
Let $D$ be a derivation of a real Lie algebra $\mathfrak n$.
The derivation $D$ may be decomposed into commuting components $D = D_{\mathsf {ss,r}} + D_{\mathsf {ss,i}} + D_{\mathsf n}$,
where $D_{\mathsf {ss,r}}$ is semisimple with a real spectrum, $D_{\mathsf {ss,i}}$ is semisimple with purely imaginary spectrum, and $D_{\mathsf n}$ is nilpotent, all being derivations (\cite[Corollary 2.6]{LDGdilation}).
The real shadow of $\mathfrak s = \mathfrak n \rtimes_D \mathbb R$ is defined as $\mathfrak s_0 = \mathfrak n \rtimes_{(D_{\mathsf {ss,r}} + D_{\mathsf n})} \mathbb R$.
\end{definition}

Heintze groups with a real shadow of Carnot type may be characterized geometrically by the fact that the conformal gauge on their boundary at infinity minus the focal point contains a geodesic metric, indeed even a subRiemannian one.

\subsection{Sol-type groups}
\label{subsec:soltype-definition}

We define below a class of solvable groups, the most prominent of which is the three-dimensional group SOL.

\begin{definition}[Sol-type]
\label{def:soltype}
Let $N_1, N_2$ be a pair of simply connected nilpotent Lie groups. 
Let $\lambda >0$.
Let $D_1, D_2$ be a pair of derivations of $\mathfrak n_1$ and $\mathfrak n_2$, respectively, so that $\mathfrak n_1 \rtimes_D \mathbb R$ and $\mathfrak{n}_2 \rtimes_D \mathbb R$ are the Lie algebras of two Heintze groups $S_1$ and $S_2$, i.e.,
the real parts of the eigenvalues of $D_1, D_2$ are positive and they are normalized so that the smallest ones of each  have real parts equal to one.
The derivation $D = D_1 \oplus (- \lambda D_2)$ acts on the Lie algebra $\mathfrak{n_1} \times \mathfrak{n}_2$ and the corresponding semi-direct product \[S = (N_1 \times N_2) \rtimes \mathbb R\]
is called a {\em Sol-type group}.
\end{definition}

A Sol-type group is unimodular if and only if $\Re \operatorname{tr}(D_1) = \lambda \Re \operatorname{tr}(D_2)$ (which does not depend on $D_1$ and $D_2$ chosen).

Similar to SOL, the group $G$  is foliated by the left cosets of $S_i = N_i \rtimes \mathbb R$.  Note that  $S_2$ is a ``upside down'' Heintze group, while $S_1$ is   right side up. See Figure~\ref{fig:soltype}.

\begin{center}
\begin{figure}
\begin{tikzpicture}[line cap=round,line join=round,>=angle 45,x=0.4cm,y=0.4cm]
\clip(-9,-10) rectangle (10,7);
\draw (-6,0)-- (-6,-10);
\draw (-6,0)-- (0,6);
\draw (0,6)-- (8,2);
\draw (-6,-10)-- (0,-4);
\draw (8,-8)-- (0,-4);
\draw (8,2)-- (8,-8);
\draw (0,6)-- (0,-4);
\draw (-5.91,0.32) node[anchor=north west] {$ S_1 $};
\draw (6.42,2.03) node[anchor=north west] {$ S_2 $};
\draw (1,-8.03) node[anchor=north west] {$ N $};
\draw [color=black, line width=1pt,domain=-5:-1, samples = 50] plot(\x,{-(\x+3)*(\x+3)+\x +1});
\draw [color=black, line width=1pt,domain=1.5:6, samples = 50] plot(\x,{(\x-4)*(\x-4)-\x/2 -2});
\draw [dash pattern=on 3pt off 3pt] (0,0)-- (-6,-6)-- (2,-10)-- (8,-4) -- cycle;
\draw [->] (0,0) -- (-1,-1) node[below]{$\mathfrak n_1$};
\draw [->] (0,0) -- (2,-1) node[above]{$\mathfrak n_2$};
\end{tikzpicture}
\caption{Sketch view of a Riemannian Sol-type group and two geodesics. Note that we do not assume that $\mathfrak n_1 \perp \mathfrak{n}_2$.}
\label{fig:soltype}
\end{figure}
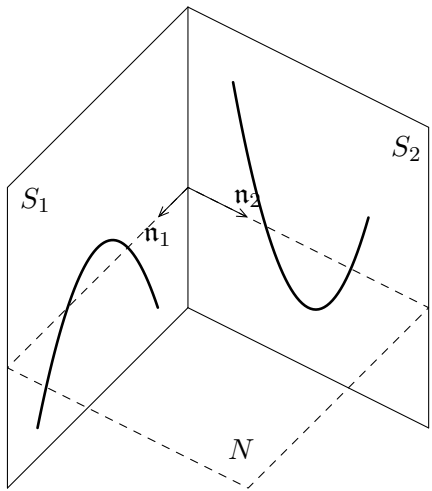

\end{center}

\subsection{Height}

\begin{definition}
Let $S = N \rtimes_D \mathbb R$ be a Heintze group as in Definition \ref{def:heintze-group}.
The projection $h\colon S \to \mathbb R$ is called the {\em height function} of $S$.
\end{definition}

\begin{definition}
Let $G = N \rtimes_D \mathbb R$ be a Sol-type group as in Definition \ref{def:soltype}.
The projection $h\colon S \to \mathbb R$ is called the {\em height function} of $G$.
\end{definition}

\section{Left-invariant Riemannian metrics on Heintze Groups}
\label{sec:heintze}

In this section we show that every two left-invariant Riemannian metrics on an Heintze group $S=N\rtimes \mathbb R$ are roughly similar through the identity map, see Theorem~\ref{main-heintze}.  
   
\begin{lemma}
\label{lem:perp-section}
Let $S$ be a simply connected solvable Lie group and assume that $N=[S,S]$ has codimension $1$ in $S$. For every left-invariant Riemannian metric $g$ on $S$, there exists a one-parameter subgroup $c \colon  S/N \to S$ that is a geodesic such that $\dot c(0) \perp \mathfrak n$ and $\pi \circ c$ is the identity on $S/N$, if $\pi \colon S \to S/N$ denotes the associated projection.
\end{lemma}

\begin{proof}
Let $\nabla$ be the Levi-Civita connection of the left-invariant metric $g$ on $S$. 
Then by Koszul's formula for the Levi-Civita connection (see e.g. (5.3) in \cite{MilnorCurLie}), for every $X,Y \in \mathfrak s$, $\nabla_X Y = \frac{1}{2} \left( \operatorname{ad}_X Y - \operatorname{ad}^{\ast}_X Y - \operatorname{ad}^{\ast}_Y X \right)$, where $\operatorname{ad}^{\ast}_X$ is such that $g(\operatorname{ad}^{\ast}_X Y, Z) = g(Y, \operatorname{ad}_X Z)$ for all $Y,Z$. 
It follows that the one-parameter subgroup $c$ generated by $T$ with $T \in \mathfrak{n}^{\perp}$ and $g(T,T) = 1$ is a geodesic of $g$, since $\nabla_{\dot{c}}{\dot {c}} = 0$.
\end{proof}
   
  The key of the proof   of    Theorem~\ref{main-heintze}   is Lemma~\ref{distance in GH} from the previous section and the fact that for every two one-parameter subgroups $c_1$, $c_2$ of $S$ not contained in $N$, every left coset of $c_1(\mathbb R)$ is  at  bounded distance from a unique left coset of $c_2(\mathbb R)$, see Lemma~\ref{1parameter}.   
 
 Let $g_1, g_2$ be two left-invariant Riemannian metrics on a Heintze group $S$. Let $c_1$ and $c_2$ be the one-parameter subgroups associated to $g_1$ and $g_2$ respectively by Lemma~\ref{lem:perp-section}.

  In the case when $c_1$ and $c_2$ have the same image, the rest of the proof of Theorem~\ref{main-heintze} is quite simple. We shall treat this case in the next section; afterwards we consider the general case.
  
  \begin{proof}[Proof of Theorem~\ref{main-heintze} when $c_1$ and $c_2$ have the same image]

  {\color{black} Observe that the height map  $S \to \mathbb R$ is $1$-Lipschitz, where we equip $S$ with $d_i$ and $S/N$ with the Hausdorff distance $\mathrm{Hausdist}_{d_i}$ for $i=1,2$.
  From now on we decompose $S$ topologically as a product $N \times \mathbb R$ where $c_1(t) = (1_N,t)$ for all $t \in \mathbb R$, and for all $n \in N$ we denote $\mathfrak c_n$ the curve $\mathfrak c_n(t)=(n,t)$. 
  By rescaling the metric $g_1$ and $g_2$ we may assume that $c_1=c_2$, and that they are unit speed geodesics for $d_1$ and $d_2$.
  It follows from the normalization convention that for $i=1,2$, $\mathrm{Hausdist}_{d_i}$ on $S/N$ is also the standard absolute value on $\mathbb R$. A useful consequence is that if two subsets are at $d_i$-Hausdorff distance bounded by $H$ for some $i$, then so are their maximal heights also differ by $H$.}
  
  Let $C$ be the constant from Lemma~\ref{distance in GH} for both $d_1$ and $d_2$.
   We shall show that the identity map $\text{Id}: (S, d_1)\rightarrow (S, d_2)$ is a rough isometry.  
   Let $x=(n, t), \tilde x=(\tilde n, \tilde t)\in S$.    
Our  assumption implies that  the curves 
$\mathfrak c_n$ and $\mathfrak c_{\tilde n}$ are unit speed minimizing geodesics with respect to both $d_1$ and $d_2$.   
  Because of Lemma~\ref{distance in GH},
  for each  $i=1, 2$ exists  $t_i$  such that
   the path $\beta_i:=x \mathfrak c_n(t_i)\cup \mathfrak c_n(t_i) \mathfrak c_{\tilde n}(t_i)\cup \mathfrak c_{\tilde n}(t_i)y$  is a $(1, C)$-quasi-geodesic in $(S, d_i)$ from $x$ to $y$.    Since the identity map $(S, d_1)\rightarrow (S, d_2)$ is biLipschitz,  the path $\beta_2$ is an  $(L,A)$-quasi-geodesic in $(S, d_1)$ from $x$ to $y$, where $L$, $A$ depend only on $d_1$ and $d_2$.   By the Morse Lemma, the Hausdorff distance between $\beta_1$ and $\beta_2$ in $(S, d_1)$ is bounded above by  a constant depending only on $d_1$ and $d_2$.  Comparing heights we see
     that 
     $|t_1-t_2|$   is bounded above by a constant depending only on $d_1$ and $d_2$.  
      Finally Lemma~\ref{distance in GH}  implies that $|d_1(x,y)-d_2(x,y)|$ is 
  bounded above by a constant depending only on $d_1$ and $d_2$.   This finishes the proof  of Theorem~\ref{main-heintze}  
    when $c_1$ and $c_2$ have the same image.
  \end{proof}
 \subsection{The general case: $c_1$ and $c_2$ might have different images}
 
 In order to consider the general case in the proof  of Theorem~\ref{main-heintze}, we need  
 the  following lemma. We shall  abbreviate the image of $\mathbb R$  under a one-parameter subgroup $c \colon \mathbb R\rightarrow S$ by $c$.  

\begin{lemma}\label{1parameter}
Let $S$ be a Heintze group with derived subgroup $N$. 
Equip $S$ with a left-invariant {\color{black} Riemannian} metric $g$.
For every two one-parameter subgroups $c_1$, $c_2$ of $S$ not contained in $N$, there is a positive number $C$ 
(depending   on $c_1, c_2$ and $g$) 
such that for every $s_1 \in S$, there is a unique left coset $s_2c_2$ of $c_2$, with $s_2\in S$,
   such that
\[ \mathrm{Hausdist}_d(s_1c_1, s_2c_2)\le C,\] 
where $\mathrm{Hausdist}_d$ denotes the Hausdorff distance with respect to the distance $d$ on $S$ determined by $g$.  
\end{lemma}

\begin{proof}
Let $g_2$ be a left-invariant Riemannian metric on $S$ such that $c_2$ is normal to $N$ with respect to $g_2$, and denote $d_2$ the associated Riemannian distance.
For every $s_1 \in S$, the curve $s_1c_1$ is an $(L,C)$-quasi-geodesic in $(S, d_2)$ for some constants $L$, $C$ depending only on $g$ and $g_2$.  By the Morse Lemma,   there is  a  complete geodesic $\gamma$ in $(S, d_2)$ such that $\mathrm{Hausdist}_{d_2}(s_1c_1, \gamma)\le C_1$ for some constant $C_1$   depending only on $g$ and $g_2$.  
Since $s_1c_1$ intersects all the horospheres centered at the focal point, so does $\gamma$ {\color{black} (Indeed, $h(s_1 c_1)$ and $h(\gamma)$ are both intervals of $\mathbb R$ at bounded Hausdorff distance from each other, so if one of them is $\mathbb R$ then the other one as well}).
  {\color{black} We see that the limit points of $\gamma$
in  $\partial S$ are the focal point   and some $n\in N$.  On the other hand, there is a left coset $s_2c_2$  with the same limit points.  Since both $\gamma$ and $s_2c_2$ are geodesics in $(S, d_2)$,  their Hausdorff distance is bounded above by a constant $H$ depending only on $d_2$.    Hence 
      $\mathrm{Hausdist}_{d_2}(s_1c_1, s_2c_2)\le C_1 +H$ for some left coset $s_2c_2$ of $c_2$. 
      The lemma follows  as all the left-invariant Riemannian metrics are biLipschitz with respect to each other.   
      Since two different cosets $s_2 c_2$ and $s'_2 c_2$ have infinite Hausdorff distance, we have uniqueness.}
\end{proof}

\begin{proof}[Proof of Theorem~\ref{main-heintze} in the general case]  
 For each $i \in \{1,2 \}$, let $g_i$ be a left-invariant Riemannian metric on $S$ and $d_i$ the distance on $S$ determined by $g_i$.  We need to show that 
    the identity map 
 $(S, d_1)\rightarrow (S, d_2)$ is a rough similarity.  
  Let $c_i$ be a $g_i$-geodesic section of $\pi \colon S \to S/N$ with $c_i(+\infty)$ equal to the focal point for all $i$. The composition $h\circ c_i: \mathbb R\rightarrow \mathbb R$ is the identity map.  
 After rescaling the metric $g_i$ if necessary we may further assume that 
   $c_i$ is a unit  speed geodesic in $(S, d_i)$.
   We shall show that the identity map 
   $(S, d_1)\rightarrow (S, d_2)$ is a rough   isometry.  
By symmetry it suffices to show that there is a constant $C$ such that 
 $d_1(x,y)\le d_2(x,y)+C$ for every $x, y\in S$. 

  By Lemma~\ref{distance in GH}, there are points $x'\in xc_2$, $y'\in yc_2$   such that  $d_2(x', y')\le C'$ and 
      \begin{equation}\label{A0}
      |d_2(x,y)-(d_2(x,x')+d_2(y',y))|\le C',
      \end{equation}
         where $C'$ depends only on $d_2$.   
  Since $\text{Id}: (S, d_1)\rightarrow (S, d_2)$ is $L$-biLipschitz for some $L\ge 1$, 
    we have 
        \begin{equation}\label{A1}d_1(x',y')\le LC'.
        \end{equation}  
     By Lemma~\ref{1parameter} there are left cosets   $\alpha$, $\beta$ of $c_1$ such that $\mathrm{Hausdist}_{d_1}(\alpha, xc_2)\le C$, 
 $\mathrm{Hausdist}_{d_1}(\beta, yc_2)\le C$, where $C$ is a constant depending only on $d_1$, $d_2$. 
         
Considering the height function $h$, we take      $\tilde x$ and $\tilde{x}'$ to be points on  $\alpha$    satisfying $h(\tilde x)=h(x)$, $h(\tilde x')=h(x')$, see Figure~\ref{fig:ThmA}. 
      Similarly  
 let $\tilde y$ and $\tilde{y}'$ be points on  $\beta$   satisfying $h(\tilde y)=h(y)$, $h(\tilde y')=h(y')$.
    We claim that we have 
    \begin{equation}\label{A2}
    d_1(z, \tilde z)\le 2C,\qquad \text{ for }z\in \{x, x', y, y'\} \text{ and the respective } \tilde z.
    \end{equation}
      Indeed, the $d_1$ distance from $z$ to the appropriate left coset of $c_1$ is at most $C$, so that the height of the nearest-point projection of $z$ on this left coset differs at most $C$ from that of $\tilde z$.
 
        We have the bounds
  \begin{align*}
   d_1(x, y) & \le d_1(x, \tilde x)+d_1(\tilde x, \tilde {x}')
   +d_1(\tilde{x}', x') 
   +d_1(x', y') \\
  & \qquad +d_1(y', \tilde{y}')+d_1(\tilde{y}', \tilde y)+d_1(\tilde y, y)\\
  &\le 8C+LC'+d_1(\tilde x, \tilde {x}')+d_1(\tilde{y}', \tilde y)\\
  &=8C+LC'+d_2(x, {x'})+d_2({y'},  y)\\
  &\le 8C+LC'+C'+d_2(x, y),  
  \end{align*}
  where we used the following arguments:
  In the first line, we used the triangle inequality.
  In the second line, we used \eqref{A1} and \eqref{A2}.
  In the third line, we used that
  $d_1(\tilde x, \tilde {x}') = \vert h(\tilde x) - h( \tilde x')\vert = \vert h(x) - h(x') \vert = d_2(x,x')$ and similarly, $d_1(\tilde y, \tilde {y}') = d_2(y,y')$.
  In the fourth line, we used \eqref{A0}. 
  \end{proof} 
  
  \begin{figure}
  \begin{tikzpicture}[line cap=round,line join=round,>=triangle 45,x=0.7cm,y=0.7cm]
\clip(-5,-4.5) rectangle (5,4.5);
\draw(0,0) circle (4);
\draw [dash pattern=on 4pt off 4pt] (0,2.75) circle (1.25);
\draw [dash pattern=on 4pt off 4pt] (0,1.74) circle (2.26);
\draw [dash pattern=on 4pt off 4pt] (0,1.06) circle (2.94);
\draw [shift={(5.15,4)}] plot[domain=3.14:4.46,variable=\t]({1*5.15*cos(\t r)},{1*5.15*sin(\t r)});
\draw [shift={(-13.31,4)}] plot[domain=-0.58:0,variable=\t]({1*13.31*cos(\t r)+0*13.31*sin(\t r)},{1*13.31*sin(\t r)});
\draw [shift={(2.57,1.99)}] plot[domain=2.48:5.13,variable=\t]({1*3.26*cos(\t r)},{3.26*sin(\t r)});
\draw [shift={(4.79,-1.44)}] plot[domain=2.29:3.41,variable=\t]({1*7.25*cos(\t r)},{1*7.25*sin(\t r)});
\fill [color=black] (0,4) circle (1.5pt) node[anchor=south] {$\omega$};
\fill [color=black] (-1.24,-1.61) circle (1.5pt)node[anchor=north west] {$x$}; 
\fill [color=black] (1.67,0.21) circle (1.5pt) node[anchor=north] {$y$}; 
\fill [color=black] (-0.23,1.51) circle (1.5pt) node[anchor=north west] {$x'$}; ;
\fill [color=black] (0.58,1.63) circle (1.5pt) node[anchor=south] {$y'$}; 
\fill [color=black] (-0.68,1.69) circle (1.5pt)
node[anchor=north east] {$\widetilde{y'}$}; 
\fill [color=black] (-1.24,2.58) circle (1.5pt)node[anchor=east] {$\widetilde{x'}$}; 
\fill [color=black] (-2.41,-0.62) circle (1.5pt) node[anchor=east] {$\widetilde{x}$}; 
\fill [color=black] (0.45,-0.49) circle (1.5pt) node[anchor=north] {$\widetilde{y}$}; 

\draw (0,0) node[stuff_fill, anchor=center] {$\beta$};
\draw (-2.5,-2) node[stuff_fill, anchor=center] {$\alpha$};

\draw (1,0.8) node[stuff_fill, anchor=center] {$yc_2$};
\draw (-1.5,-2.5) node[stuff_fill, anchor=center] {$xc_2$};
\end{tikzpicture}
\caption{Main objects in the proof of Theorem~\ref{main-heintze} in the general case, in the hyperbolic disk model with a focal point $\omega$. From the point of view of $d_2$, $d_1$-geodesics appear in the form of hypercircles.}\label{fig:ThmA}
  \end{figure}
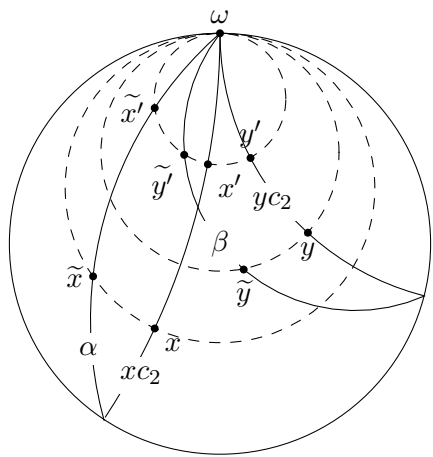

\section{Left-invariant Riemannian metrics on  Sol-type groups}
\label{sec:soltype}

In this section we show that every two left-invariant Riemannian metrics with associated distances $d_1$ and $d_2$ on a  Sol-type group   $G$    are roughly similar through
 the identity map, see Theorem~\ref{main-soltype}.  Notation here is as in \S~\ref{subsec:soltype-definition}, especially $G=(N_1\times N_2)\rtimes \mathbb R$, $S_1 = N_1 \rtimes \mathbb R$ and $S_2 = N_2 \rtimes \mathbb R$.

\begin{figure}
    \begin{center}
\begin {tikzpicture}[-latex ,auto ,node distance =1 cm and 3.5cm ,on grid ,
semithick , state/.style ={ rectangle , 
minimum width =1 cm}, state1/.style ={ rectangle ,
 draw,minimum width =1 cm}]
\node[state] (P) {\bf Lemma~\ref{projection}};
\node[state] (C) [below =of P] {\bf Lemma~\ref{claim}};
\node[state1] (Q) [below = of C] {\bf Theorem~\ref{soldistancetheorem}};
\node[state] (H) [below = of Q] {\bf Corollary~\ref{solpath}};
\node[state1] (T) [below left = of H] {{\bf Theorem~\ref{main-soltype}}, general case.};
\node[state] (H2) [left =of P] {\bf Lemma~\ref{path-outside}};
\node[state] (H1) [left =of H2] {\bf Lemma~\ref{comparison-horosphere}};
\node[state] (SO) [left=of Q] {\bf Lemma~\ref{Sol-1parameter}};
\node[state] (O) [below=of H1] {\bf Lemma~\ref{1parameter}};

\path (P) edge [bend left =0] node[left]{\S~\ref{subsec:43}} (C);
\path (C) edge [bend left =0]   node[right]{\S~\ref{subsec:43}} (Q) ;
\path (Q) edge [bend left =0]   node[right]{\S~\ref{another}} (H);
\path (H) edge (T) ;
\path (H1) edge  node[above]{\S~\ref{sec:preliminary}} (H2)  ;
\path (H2) edge (C)  ;
\path (O) edge  node[above]{\S~\ref{soldistance}} (SO);
\path (SO) edge  (T);

\draw (-3.2,-3.5) node {\S~\ref{soldistance}}; 

\end{tikzpicture}
\end{center}
    \caption{Scheme of the proof of Theorem~\ref{main-soltype}}
    \label{fig:proof-main-soltype}
\end{figure}
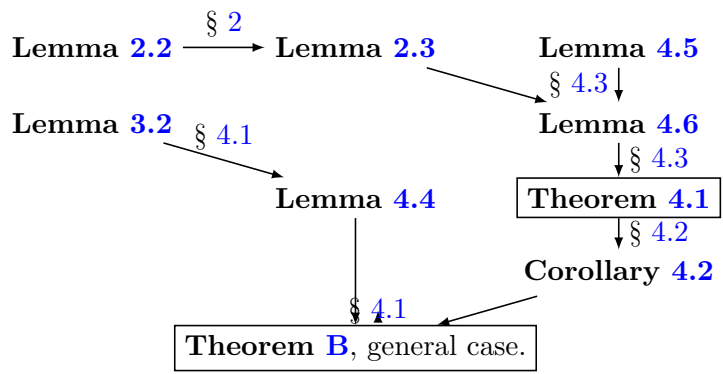
 
 The general strategy is the same as for the Heintze groups: we first establish the statement for those  pairs of Riemannian metrics for which $\mathfrak n_1 \times \mathfrak{n}_2$ have the same orthogonal complement in $\mathfrak g$, then the general case.   
To establish the special case we need to find an estimate for the distance function, see Theorem~\ref{soldistancetheorem}. 

\subsection{Distances on Sol-type groups  and the proof of Theorem~\ref{main-soltype}}\label{soldistance}

In this subsection we state a result giving an estimate of distance in Sol-type groups and use  it to prove Theorem~\ref{main-soltype}.   The estimate  itself will be established later.  


Let $g$ be a left-invariant  Riemannian metric on $G$. By Lemma~\ref{lem:perp-section} we choose a geodesic section for $G \to \mathbb R$, and then assume without loss of generality that as a set, $G=N_1\times N_2\times \mathbb R$, where $\mathbb R$ direction is perpendicular to both $N_1$ and $N_2$ with respect to $g$. 
   As in the case of Heintze groups, this  assumption implies that  for every $x\in N_1$, $y\in N_2$,  the curve $\gamma_{x,y}(t)=(x,y, t)$ ($t\in \mathbb R$) is a minimizing constant speed geodesic.   These will be called vertical geodesics.  

 We define several maps. The map $h: G\rightarrow \mathbb R$,  $h(x,y,t)=t$, will be called the height function  of $G$   and $t$ will be called the height of the point $(x,y,t)$. 
  The ``projections''  $\pi_i: G \rightarrow S_i$ are defined by $\pi_1(x,y,t)=(x,t)$, $\pi_2(x,y,t)=(y,t)$.  
We emphasize that the maps $\pi_i$ are not nearest point  projections.     However,   they are Lie group homomorphisms and so are Lipschitz  
   with 
     respect to left-invariant Riemannian  metrics, see Lemma~\ref{projection}.

  By  rescaling the metric $g$ we may assume that the vertical geodesics are unit speed geodesics.     Denote by $d$ the distance on $G$ determined by $g$.  We  identify $S_1$ with 
   $N_1\times \{0\}\times \mathbb R\subset G$ and $S_2$  with $\{0\}\times N_2\times \mathbb R\subset G$.  
 For $j=1,  2$, let $g^{(j)}$ be the Riemannian metric on $S_j$ induced by $g$ and $d^{(j)}$ the associated distance on $S_j$.   
 
 
Define a ``distance'' $\rho: G\times G\rightarrow [0, \infty)$ on $G$  by: 
 \begin{equation}\label{disformula}
 \rho(p,q)=d^{(1)}(\pi_1(p), \pi_1(q))+d^{(2)}(\pi_2(p), \pi_2(q))-|h(p)-h(q)|.
 \end{equation}
  It turns out   that the distance $d$ on $G$ differs from $\rho$  by a bounded constant:
  
  \begin{theorem}\label{soldistancetheorem}
    Let $G$, $d$ and $\rho$ be as above. Then there exists a constant $C>0$ such that
  $|d(p,q)-\rho(p,q)|\le C$ for all $p, q\in G$. 
   \end{theorem}
  
   As a consequence of the proof of Theorem~\ref{soldistancetheorem} we have

  \begin{corollary}\label{solpath}
     Let $G$ be a Sol-type group, $g$ a left-invariant Riemannian metric and $d$ the associated distance.  Then there is a constant $C>0$ with the following property. 
         Denote by $c$ the one-parameter subgroup of $G$ that is perpendicular to 
    $N$ at $e$   with respect to $g$. 
    For every   $x,y\in G$  with $h(x)\le h(y)$,    there exist   three left cosets $\beta_j$ ($j=1, 2,   3$)   of $c$   with $x\in \beta_1$, $y\in \beta_3$
     and points $x_1\in \beta_1$, $z_1, z_2\in \beta_2$ and $y_2\in \beta_3$ satisfying:
\begin{enumerate}
\item 
$|d(x,y)-(d(x, x_1)+d(z_1, z_2)+d(y_2,y))|\le C$;
\item 
 $d(x_1, z_1)\le C$, $d(z_2, y_2)\le C$;
\item 
 $h(x_1)=h(z_1)\le h(x)$, $h(z_2)=h(y_2)\ge h(y)$.  
\end{enumerate}     
  \end{corollary}

          \begin{center}
          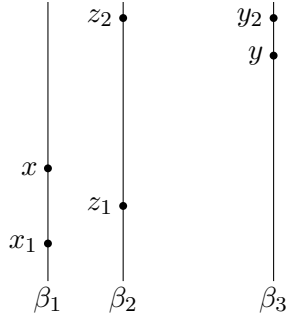
\begin{figure}
          \begin{tikzpicture}[x=1cm,y=0.5cm]
          
          \clip (-2,-4) rectangle (5,4.4);
          
          \fill (-1,0) circle(1.5pt) node[left] {$x$};
          \fill (-1,-2) circle(1.5pt) node[left] {$x_1$};
          \fill (0,-1) circle(1.5pt) node[left] {$z_1$};
          \fill (0,4) circle(1.5pt) node[left] {$z_2$};
          \fill (2,4) circle(1.5pt) node[left] {$y_2$};
          \fill (2,3) circle(1.5pt) node[left] {$y$};
          
          \draw (-1,-3) -- (-1,5);
          \draw (0,-3) -- (0,5);
          \draw (2,-3) -- (2,5);
          
          \draw (-1,-3.5) node[anchor=center] {$\beta_1$};
          \draw (0,-3.5) node[anchor=center] {$\beta_2$};
          \draw (2,-3.5) node[anchor=center] {$\beta_3$};
          
          \end{tikzpicture}
          \caption{The left cosets of $c$ in Corollary~\ref{solpath}.}
          \end{figure}

          \end{center}

  \begin{remark}  Tom Ferragut has a result similar to Theorem~\ref{soldistancetheorem}, see Corollary 4.17 of \cite{F}.  These two results have overlap but do not imply each other.  The result in \cite{F} is for horospherical products $X\bowtie Y$  of Gromov Busemann spaces $X$, $Y$.   
  On one hand, horospherical products are more general than Sol-type groups. 
  On the other hand,   
   the factors $X$ and $Y$ in a horospherical product are ``perpendicular'' in some sense, 
   while  $N_1$ and $N_2$ in a Sol-type group  is not assumed to be perpendicular to each other  with respect to the metric $g$ (without loss of generality the direction of the $\mathbb R$ factor is perpendicular to both $N_1$ and $N_2$, but we do not assume that $N_1$ and $N_2$ are perpendicular).

  \end{remark}
  
  \begin{proof}[Proof of Theorem~\ref{main-soltype} assuming Theorem~\ref{soldistancetheorem} in the case of equal vertical gedesics]   
  
 One may decompose $G$ as a product, $G = N_1 \times N_2 \times \mathbb R$, in such a way that the direction of the
  $\mathbb R$ factor is perpendicular to both $N_1$ and $N_2$ with respect to  $g_1$ and $g_2$.

  After rescaling we may further assume that     vertical geodesics have  unit speed with respect to both $g_1$ and $g_2$.  We shall show that the identity map
     $\text{Id}: (G,  d_1)\rightarrow (G, d_2)$ is a rough isometry.

   Let $g^{(j)}_i$ be the Riemannian metric on $S_j$  induced by $g_i$, and let $d^{(j)}_i$ be the associated distance on $S_j$.  Also let $\rho_i$ be the ``distance'' (see  (\ref{disformula}))  on $G$  corresponding to $g_i$.    By Theorem 
  ~\ref{soldistancetheorem}  there is a constant $C>0$ such that 
    $|d_i(p, q)-\rho_i(p,q)|\le  C$  for   all  $p, q\in G$.  
    On the other hand,   since the vertical geodesics have unit speed in $(S, g_i)$,     Theorem~\ref{main-heintze}  implies 
   $|d^{(j)}_1(\pi_j(p), \pi_j(q))-d^{(j)}_2(\pi_j(p), \pi_j(q))|\le C'$ for some constant $C'\ge 0$ and all $p, q\in G$.   It follows from the definition of $\rho_i$ that 
     $|\rho_1(p,q)-\rho_2(p, q)|\le 2C'$    and so 
      $|d_1(p, q)-d_2(p, q)|\le 2C+2C'$ for all $p, q\in G$. \qedhere
      
      \end{proof}

      For the general case, we need an analogue of 
        Lemma~\ref{1parameter} for Sol-type groups.

      \begin{lemma}\label{Sol-1parameter}
      Let $c, \tilde c: \mathbb R\rightarrow G=(N_1\times N_2)\rtimes \mathbb R$ be  one-parameter subgroups of $G$ not contained in $(N_1\times N_2)\times \{0\}\subset G$. 
       Let $g$ be any left-invariant Riemannian metric on $G$.
          Then there is a constant $C$ depending only on $g$ and $c, \tilde c$ with the following property.  For any left coset $pc$ of $c$, there is a unique  left coset  $q\tilde c$ of $\tilde c$ 
            such that 
          $\mathrm{Hausdist}_d(pc, q\tilde c)\le C$.
      
      \end{lemma}
      
      \begin{proof}   
      Since $c, \tilde c$ are not contained in $(N_1\times N_2)\times \{0\}$,   the compositions $h\circ c$ and $h\circ \tilde c$ are automorphisms of $\mathbb R$.  
        By  composing $c$, $\tilde c$  with  suitable  automorphisms of  $\mathbb R$  we may assume $h\circ c(t)=t$ and $h\circ \tilde c(t)=t$ for $t\in \mathbb R$.  
        The one-parameter subgroups $c$ and $\tilde c$ now have the expressions:
       $c(t)=(a_1(t), a_2(t),  t)$, $\tilde c(t)=(\tilde a_1(t), \tilde a_2(t), t)$ for some functions $a_1, \tilde a_1: \mathbb R \to N_1$, $a_2, \tilde a_2: \mathbb R \to N_2$.

        Since the distance $d$ is left invariant, we may assume $p=e$ and so $pc=c$.  
         As  the projections $\pi_1$ and $\pi_2$ are group homomorphisms, the compositions $\pi_i\circ c$, $\pi_i\circ \tilde c$ are one-parameter subgroups of $S_i$ that are not contained in $N_i\times \{0\}\subset S_i$.    
         By Lemma~\ref{1parameter}, there exist   constants $C_1, C_2>0$ depending only on $c$, $\tilde c$ and $g$,  $(n_1, 0)\in S_1$ and $(n_2, 0)\in S_2$ such that $\mathrm{Hausdist}_{d_1}(\pi_1(c), (n_1,0)\pi_1(\tilde c))\le C_1$ and 
         $\mathrm{Hausdist}_{d_2}(\pi_2(c), (n_2,0)\pi_2(\tilde c))\le C_2$, where $d_{i}$ denotes the distance on $S_i$   induced by the restriction of $g$ to $S_i$. 
           Here we identify $S_1$ with $N_1\times\{0\}\times \mathbb R\subset G$ and similarly $S_2$ with $\{0\}\times N_2\times \mathbb R\subset G$.  
         Clearly we have $d(x,y)\le d_i(x,y)$  for any $x,y\in S_i$.   It follows that  $d_1((a_1(t),t), (n_1\tilde a_1(t),t))\le 2C_1$ and 
          $d_2((a_2(t),t), (n_2 \tilde a_2(t),t))\le 2C_2$ for all $t\in \mathbb R$.    Set $q=(n_1, n_2, 0)\in G$. 
         We next show that $d(c(t), q\tilde c(t))\le  2C_1+2C_2 $  for all $t\in \mathbb R$   and so   $\mathrm{Hausdist}_d(c, q\tilde c)\le 2C_1+2C_2$. 
 \begin{align*}
 &d(c(t), q\tilde c(t))\\
 &= d((a_1(t), a_2(t), t), (n_1, n_2, 0)(\tilde a_1(t), \tilde a_2(t),t))\\
 &\le d((a_1(t), a_2(t), t),  (n_1\tilde a_1(t), a_2(t), t))+d( (n_1\tilde a_1(t), a_2(t), t),   (n_1, n_2, 0)(\tilde a_1(t), \tilde a_2(t),t))\\
 &= d((0, a_2(t),0)(a_1(t), 0,t),  (0, a_2(t),0)(n_1\tilde a_1(t),0,  t))+\\
&+ d((n_1\tilde a_1(t), 0,0)(0, a_2(t), t),  (n_1\tilde a_1(t), 0,0)  (0, n_2\tilde a_2(t),  t))\\
 &=d((a_1(t), 0,t),  (n_1\tilde a_1(t),0,  t))+ d((0, a_2(t), t), (0, n_2\tilde a_2(t),  t)) \\
& \le d_1(a_1(t), t), (n_1\tilde a_1(t), t))+d_2((a_2(t), t), (n_2\tilde a_2(t), t))\\
 &\le 2C_1+2C_2.
 \end{align*}

      \end{proof}

      \begin{proof}[Proof of Theorem~\ref{main-soltype}  in the general case]

 Let $g, \tilde g$ be  left-invariant Riemannian metrics on $G$ and $d, \tilde d$ the distances on $G$ determined by $g$, $\tilde g$  respectively.  We need to show that 
    the identity map 
 $(G, d)\rightarrow (G,   \tilde d)$ is a rough similarity.  

  Let $c$    be  the one-parameter subgroup of $G$ whose tangent vector at $e$ is $g$-perpendicular to $(N_1\times N_2)\times \{0\}$  and $h(c(t)) = t$.
   After rescaling the metric   $g$  if necessary we may assume that 
     $c$ is a  unit speed geodesic with respect to $g$.  This normalization implies  that if $\beta$ is a left coset of $c$ and $p_1, p_2\in \beta$, then $d(p_1, p_2)=|h(p_1)-h(p_2)|$.  
       Similarly let $\tilde c$ be the normalized one-parameter subgroup of $G$ corresponding to $\tilde g$ such that $h(\tilde c (t)) = t$.   
       We observe that,   
           if $\beta$ is a left coset of $c$ and $\tilde \beta$ is a left coset of $\tilde c$, and $p_1, p_2\in \beta$, $\tilde p_1, \tilde p_2\in \tilde \alpha$ with $h(\tilde p_1)=h(p_1)$, $h(\tilde p_2)=h(p_2)$, then  
     $d(p_1, p_2)=|h(p_1)-h(p_2)|=\tilde d(\tilde p_1, \tilde p_2)$.   
    We shall show that the identity map 
   $(G, d)\rightarrow (G, \tilde d)$ is a rough   isometry.

By symmetry it suffices to show that there is a constant $C$ such that 
 $d(p,q)\le \tilde d(p,q)+C$ for every $p, q\in G$.  
 Let $p,q\in G$. We may assume $h(p)\le h(q)$.   By  Corollary~\ref{solpath}, there are three left cosets  $\tilde\beta_i$ ($i=1, 2,3$) of $\tilde c$ with $p\in\tilde  \beta_1$, $q\in \tilde\beta_3$,
   points $\tilde p_1\in \tilde\beta_1$, $\tilde r_1, \tilde r_2\in \tilde \beta_2$,  $\tilde q_2\in \tilde \beta_3$  satisfying the following (setting $\tilde p_2:=p$, $\tilde q_1:=q$):   $h(\tilde p_1)=h(\tilde r_1)\le  h(\tilde p_2)$,  $h(\tilde r_2)=h(\tilde q_2)\ge h(\tilde q_1)$, 
    $\tilde d(\tilde p_1, \tilde r_1)<\tilde C$, $\tilde d(\tilde r_2, \tilde q_2)<\tilde C$, and $|\tilde d(p, q)-(\tilde d(p, \tilde p_1)+\tilde d(\tilde r_1, \tilde r_2)+\tilde d(\tilde q_2, q))|\le \tilde C$, where $\tilde C$ is a constant   depending only on $\tilde d$.   By Lemma~\ref{Sol-1parameter},  there are left cosets  $\beta_i$ ($i=1,2,3$) of $c$,  such that for every $x\in \beta_i$, $y\in \tilde\beta_i$ with  the  same height (that is, $h(x)=h(y)$)  we have $d(x,y)\le C$, where $C$ is a constant depending only on $d$ and $\tilde c$.     Let 
     $p_j\in \beta_1$, $r_j\in \beta_2$, $q_j\in \beta_3$  ($j=1,2$)  satisfying $h(p_j)=h(\tilde p_j)$, $h(r_j)=h(\tilde r_j)$, 
       $h(q_j)=h(\tilde q_j)$.     Since $d$ and $\tilde d$ are biLipschitz through the identity,  there is a constant $C_1$ depending only on $d$, $\tilde d$ such that 
        $d(\tilde p_1, \tilde r_1), d(\tilde r_2, \tilde q_2)\le C_1$.    Now we have
        \begin{align*}
       d(p,q) & =d(\tilde p_2, \tilde q_1) \\
       & \le d(\tilde p_2,  p_2)+d(p_2, p_1)+d(p_1, \tilde p_1)+d(\tilde p_1, \tilde r_1)+d(\tilde r_1, r_1) \\
       & \quad +d(r_1, r_2)+d(r_2, \tilde r_2)+
        d(\tilde r_2, \tilde q_2)+d(\tilde q_2, q_2)+d(q_2, q_1)+d(q_1, \tilde q_1)\\
        & \le d(p_2, p_1)+d(r_1, r_2)+d(q_2, q_1)+ 6C+2C_1\\
        &= \tilde d(\tilde p_2, \tilde p_1)+\tilde d(\tilde r_1, \tilde r_2)+\tilde d(\tilde q_2, \tilde q_1)+6C+2C_1\\
       & \le \tilde d(\tilde p_2, \tilde q_1)+\tilde C+6C+2C_1\\
        &=\tilde d(p,  q)+\tilde C+6C+2C_1.  \qedhere
        \end{align*}
 \end{proof}

\subsection{Another expression for $\rho$}\label{another}


We next start the proof of Theorem~\ref{soldistance} and Corollary~\ref{solpath}.   
    Up to an additive constant, $\rho$ admits another expression which is  more convenient for our purpose. 
We first fix some notation.  

Let $G$, $g$, $d$,  $d^{(j)}$   and $\rho$ be as in  
 Subsection~\ref{soldistance}.   We recall that $g$ is  a left-invariant  Riemannian metric on $G$ such that   the 
 $\mathbb R$ direction is perpendicular to both $N_1$ and $N_2$ with respect to $g$,   the vertical geodesics $\gamma_{x_1, x_2}$ ($x_1\in N_1, x_2\in N_2$) are unit speed minimizing geodesics,  and  
      the minimal distance between the two ``horizontal sets'' $N_1\times N_2\times \{t_1\}$  and  $N_1\times N_2\times \{t_2\}$  is $|t_1-t_2|$.


Since $S_1$, $S_2$ are Gromov-hyperbolic, there is some constant $\delta>0$ such that both $(S_1, d^{(1)})$ and $(S_2, d^{(2)})$  are $\delta$-hyperbolic.

For each $t\in \mathbb R$, let $d^{(1)}_t$ be the path metric on the set $N_1\times \{t\}\subset (S_1, d^{(1)})$.  By the geometry of Heintze groups, we know that for fixed $x_1, y_1\in N_1$,   the quantity 
$d^{(1)}_t(\gamma_{x_1}(t), 
\gamma_{y_1}(t))$ decreases exponentially as $t\rightarrow +\infty$.  
Let $t_{x_1, y_1}\in \mathbb R$ be such that  
$d^{(1)}_{t_{x_1, y_1}}(\gamma_{x_1}(t_{x_1, y_1}), 
\gamma_{y_1}(t_{x_1, y_1}))=1$.   
Define  a function $\tilde\rho_1: S_1\times S_1\rightarrow [0, +\infty)$ by:
   \begin{equation*}
   \tilde\rho_1((x_1,t), (y_1, s))=\left\{
   \begin{array}{rl}
   |t-s|+1&  \;\text{if}  \;  t_{x_1, y_1}\le \max\{t, s\}\\
   (t_{x_1, y_1}-t)+(t_{x_1, y_1}-s)+1 &  \; \text{if}\; 
     t_{x_1, y_1}> \max\{t, s\}.
     \end{array}\right.
   \end{equation*}

\begin{center}
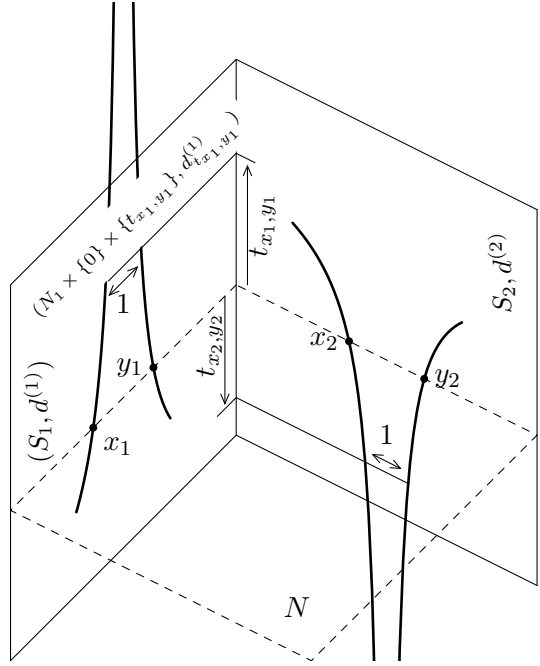
\begin{figure}
\begin{tikzpicture}[line cap=round,line join=round,>=angle 45,x=0.5cm,y=0.5cm]
\clip(-9,-10) rectangle (10,7.5);
\draw (-6,0)-- (-6,-10);
\draw (-6,0)-- (0,6);
\draw (0,6)-- (8,2);
\draw (-6,-10)-- (0,-4);
\draw (8,-8)-- (0,-4);
\draw (8,2)-- (8,-8);
\draw (0,6)-- (0,-4);
\draw (-5.91,-5) node[anchor=north west, rotate =90] {\small $ (S_1, d^{(1)}) $};
\draw (6.42,-1.03) node[anchor=north west, rotate = 90] {\small $ S_2, d^{(2)} $};
\draw (1,-8.03) node[anchor=north west] {$ N $};
\draw [color=black, line width=1pt,domain=-4.25:-1.75, samples = 50] plot(\x,{4/abs((\x+3))+\x-5});
\draw [color=black, line width=1pt,domain=1.5:6, samples = 50] plot(\x,{-4/abs((\x-4))-\x/2 +4});
\draw [dash pattern=on 3pt off 3pt] (0,0)-- (-6,-6)-- (2,-10)-- (8,-4) -- cycle;

\fill (-3.8,-3.8) circle(1.5pt) node[anchor=north west]{$x_1$};
\fill (-2.2,-2.2) circle(1.5pt) node[anchor=east]{$y_1$};

\draw (-5.5, -1) node[stuff_fill, anchor=west, rotate = 45] {\tiny $(N_1 \times \{ 0 \} \times \{ t_{x_1,y_1} \} , d^{(1)}_{t_{x_1,y_1}})$};
\draw (-3.4,0.1) -- (0,3.5) --(0.5,3.25);
\draw [<->] (-3.4,-0.3) -- (-2.6,0.5) ;
\draw (-3,-0.5) node[anchor=center]{$1$};
\draw [->] (0.3,0) -- (0.3,3.2);
\draw (0.7,1.6)  node[anchor=center, rotate = 90]{$t_{x_1,y_1}$};

\draw (-0.5,-3.5) -- (0,-3) -- (4.6,-5.3);

\draw [<->] (3.6,-4.5) -- (4.4,-4.9);

\fill (3,-1.5) circle(1.5pt) node[anchor=east]{$x_2$};
\fill (5,-2.5) circle(1.5pt) node[anchor=west]{$y_2$};

\draw (4,-4) node[anchor=center] {$1$} ;
\draw [->] (-0.3,-0.3) -- (-0.3,-3.2);
\draw (-0.7,-1.6)  node[anchor=center, rotate = 90]{$t_{x_2,y_2}$};

\end{tikzpicture}
\caption{``Metric view'' of the vertical geodesics and definition of $t_{x_1,y_1}$ in a Sol-type group. Beware that this is not a coordinate view and we equip left cosets of subgroups with path (Riemannian) metrics.}
\label{fig:soltype-metricview-verticals}
\end{figure}
\end{center}

   The quantity  $\tilde\rho_1((x_1,t), (y_1, s))$ is  (roughly) the length of a path  in $S_1$ between $(x_1, t)$ and $(y_1, s)$.    To simplify notation, denote $p=(x_1, t)$, $q=(y_1, s)$.   
   First assume $ t_{x_1, y_1}> \max\{t, s\}$.  
    Let  $\tilde \gamma_{pq}:=\gamma_{x_1}|_{[t,  t_{x_1, y_1}]}*c*\bar\gamma_{y_1}|_{[s,  t_{x_1, y_1}]}$
     be the concatenation of three paths, where
      $c:[0,1]\rightarrow N_1\times \{ t_{x_1, y_1}\}$ is a  path in the horosphere $N_1\times \{ t_{x_1, y_1}\}$
       with length $1$ from $(x_1,  t_{x_1, y_1})$ to 
       $(y_1,   t_{x_1, y_1})$.   
       Here for every curve
    $\alpha: [a,b]\rightarrow X$  in a space $X$, we will use 
    $\bar \alpha: [a,b]\rightarrow X$
      to denote the curve
        $\bar \alpha(t)=\alpha(a+b-t)$   
      with the same image as that of $\alpha$  but reverse orientation. 
         It is clear that 
        $\ell(\tilde \gamma_{pq})=\tilde\rho_1(p, q)$.  
         Next assume $ t_{x_1, y_1}\le \max\{t, s\}$. 
          Without loss of generality we may assume $t<s$.
           In this case, let 
   $\tilde \gamma_{pq}:=\gamma_{x_1}|_{[t,s]}*c$, where $c$ is a minimal length path in the horosphere
   $N_1\times \{s\}$ from $(x_1, s)$ to $(y_1, s)$. 
     It is clear that  $\tilde\rho_1(p, q)-1\le \ell(\tilde \gamma_{pq})\le \tilde\rho_1(p, q)$.

   It is easy to see that the path 
   $\tilde \gamma_{pq}$ is a  $(1, H)$ quasigeodesic between $p$ and $q$  and that  its length
   $\ell(\tilde \gamma_{pq})\le d^{(1)}(p, q)+H$, with $H \geqslant 0$ depending only on $\delta$.  Hence by stability of quasigeodesics in Gromov-hyperbolic spaces,   for every length minimizing geodesic $\gamma_{pq}$ between $p$ and $q$ the Hausdorff distance between    $\tilde \gamma_{pq}$
     and $\gamma_{pq}$   satisfies:
     \begin{equation}\label{hausdorff}
     \mathrm{Hausdist}^{(1)}(\tilde \gamma_{pq},  \gamma_{pq})\le C,
     \end{equation}
      where $C$ depends only on $\delta$.

   Similarly,  let $d^{(2)}_t$ be the path metric on the set $N_2\times \{t\}\subset (S_2, d^{(2)})$.  For fixed $x_2, y_2\in N_2$,   the quantity 
$d^{(2)}_t(\gamma_{x_2}(t), 
\gamma_{y_2}(t))$ decreases exponentially as $t\rightarrow -\infty$. 
Let $t_{x_2, y_2}\in \mathbb R$ be such that  
$d^{(2)}_{t_{x_2, y_2}}(\gamma_{x_2}(t_{x_2, y_2}), 
\gamma_{y_2}(t_{x_2, y_2}))=1$.   
Define  a function $\tilde\rho_2: S_2\times S_2\rightarrow [0, +\infty)$ by:
   \begin{equation*}
   \tilde\rho_2((x_2,t), (y_2, s))=\left\{
   \begin{array}{rl}
   |t-s|+1&  \;\text{if}  \;  t_{x_2, y_2}\ge \min\{t, s\}\\
   (t-t_{x_2, y_2})+(s-t_{x_2, y_2})+1 &  \; \text{if}\; 
     t_{x_2, y_2}< \min\{t, s\}.
     \end{array}\right.
   \end{equation*}

There is a constant $C>0$ such that 
  $$|d^{(j)}((x_j, t), (y_j, s))-\tilde \rho_j((x_j, t), (y_j, s))|\le C,   \; \;  \forall (x_j, t), (y_j, s)\in S_j.$$
  Define $\tilde \rho: G\times G\rightarrow [0, +\infty)$ by   $$\tilde \rho(p,q)=\tilde \rho_1(\pi_1(p), \pi_1(q))+
  \tilde \rho_2(\pi_2(p), \pi_2(q))-|h(p)-h(q)|.$$
  Then $\rho$ and $\tilde \rho$ differ   by at most a fixed constant. 
  
  Let $p, q\in G$ and write $p=(x_1, x_2, t)$ and $q=(y_1, y_2, s)$.  Without loss of generality we may assume that $s\ge t$.  
   We shall construct a path  $\alpha_0$  from $q$ to $p$.    We first notice that   the left cosets of $S_i$  equipped with the path metric is isometric to 
        $(S_i, d^{(i)})$.  
        Denote  $p'=( x_1, y_2, t)$.  Let 
         $\alpha_0=\tilde \gamma_{qp'}*\tilde \gamma_{p'p}$, 
          where $\tilde \gamma_{qp'}\subset N_1\times \{y_2\}\times \mathbb R\approx S_1$ is the path from $q$ to $p'$  constructed in this subsection and similarly  
   $\tilde \gamma_{p'p}\subset \{x_1\}\times N_2\times \mathbb R\approx S_2$ is the path from $p'$ to $p$. 
   We have $\tilde \rho_1((y_1, s), (x_1,t))-1\le \ell(\tilde \gamma_{qp'})\le \tilde \rho_1((y_1, s), (x_1,t))$.   Let $q'=(x_1,  y_2, s)$.   Then $\tilde\gamma_{q'p}$ is the concatenation of 
   $\bar\gamma_{x_1}|_{[t,s]}$ and $\tilde \gamma_{p'p}$. Hence we  have
       $\tilde 
       \rho_2(((x_2, t), (y_2, s))-1\le \ell(\tilde \gamma_{p'p})+|s-t|\le \tilde 
       \rho_2(((x_2, t), (y_2, s))$.  It follows that the length $\ell_0$ of $\alpha_0$ satisfies
        $|\ell_0-\tilde \rho(p,q)|\le 2$.  
    
     Since $\rho$, $\tilde \rho$, $\ell_0$   differ from each other  by a fixed constant, the following gives an expression for these quantities up to a  constant:
      for $p=(x_1, x_2,t)$, $q=(y_1, y_2, s)$:  
     \begin{equation*}
   \tilde\rho(p, q)=\left\{
   \begin{array}{rl}
   |t-s|+2&  \;\text{if}  \;  t_{x_2, y_2}\ge \min\{t, s\}\;\text{and}\;   t_{x_1, y_1}\le \max\{t, s\}  \\
   2t_{x_1, y_1}-(s+t)+2&  \;\text{if}\; t_{x_2, y_2}\ge \min\{t, s\}\;\text{and}\;   t_{x_1, y_1}\ge \max\{t, s\}  \\
  (s+t)- 2t_{x_2, y_2}+2&  \;\text{if}\; t_{x_2, y_2}\le \min\{t, s\}\;\text{and}\;   t_{x_1, y_1}\le \max\{t, s\}  \\
  2t_{x_1, y_1}-2t_{x_2, y_2}-|s-t|   +2  &  \; \text{if}\; 
     t_{x_2, y_2}\le  \min\{t, s\}  \;\text{and}\;  t_{x_1, y_1}\ge \max\{t, s\} .
     \end{array}\right.
   \end{equation*}

\begin{proof}[Proof of Corollary~\ref{solpath}]
        We first consider  the case when $\mathbb R$ is perpendicular to $N$ with respect to $g$.
        In this case   $c=\{0\}\times \{0\}\times \mathbb R$ and the left cosets of $c$ are vertical geodesics $\gamma_{x,y}$.   Let $p=(x_1, x_2, t)$ and $q=(y_1, y_2, s)$  with $t\le s$. 
            With the notation from above,  the three left cosets are 
         $\beta_1=\gamma_{x_1, x_2}$, $\beta_2=\gamma_{x_1, y_2}$, $\beta_3=\gamma_{y_1, y_2}$. The points are $p_1=(x_1, x_2, t_{x_2, y_2})$,   
          $r_1=(x_1, y_2, t_{x_2, y_2})$, $r_2=(x_1, y_2, t_{x_1, y_1})$, $q_2=(y_1, y_2, t_{x_1, y_1})$.    Notice that the length $\ell_0$ of $\alpha_0$ satisfies
           $|\ell_0-(d(p, p_1)+d(r_1, r_2)+d(q_2, q))|\le 2$.   Now the claim follows from Theorem~\ref{soldistance}  and the fact that $\ell_0$ and $\rho(p,q)$ differ by a bounded constant.
           
           Now let $g$ be an arbitrary left-invariant Riemannian metric on $G$.  Let $c$ be a one-parameter subgroup of $G$ that is perpendicular to $N$ at $e$  with respect to $g$.
                Then the above argument goes through with vertical lines replaced with left cosets of $c$,  
                   $S_1$  replaced with  
                 $(N_1\times \{0\}\times \{0\})c$, and   $S_2$  
                  replaced with   
                    $(\{0\}\times N_2\times\{0\})c$.
           \end{proof}

  \subsection{Proof of Theorem~\ref{soldistancetheorem}}
  \label{subsec:43}

Let $G$, $d$,  $d^{(j)}$   and $\rho$ be as in  
 Subsection~\ref{soldistance}.  

\begin{lemma}\label{projection}
  The map $\pi_j: (G, d)\rightarrow (S_j, d^{(j)})$ is $L$-Lipschitz for some $L\ge 1$. 
\end{lemma}

\begin{proof}  This follows from the fact that $\pi_j$ is a Lie group homomorphism. 
We shall only prove the case when $j=1$ since the 
  case for $j=2$ is similar. 
It suffices to show that the   operator  
   norm of   the tangential map $D_p\pi_1:  
  (T_p G, g)  \rightarrow (T_{\pi_1(p)}S_1,  g^{(1)})$  is  independent of the point $p$.  

It is easy to check that the following diagram commutes for every $x\in G$:
\[ \begin{tikzcd}
G \arrow{r}{\pi_1}  \arrow{d}{L_x} & 
S_1 \arrow{d}{L_{\pi_1(x)}}\\
G  \arrow{r}{\pi_1}  & S_1
\end{tikzcd} \]
This leads to the commuting diagram of tangential maps:
\[
\begin{tikzcd}
(T_xG, g)\arrow{r}{D_x\pi_1}  \arrow{d}{D_xL_{x^{-1}}}  &   (T_{\pi_1(x)}S_1,  g^{(1)})\arrow{d}{D_{\pi_1(x)}L_{\pi_1(x^{-1})}}\\
(T_e G, g)  \arrow{r}{D_e\pi_1}&  (T_e S_1,  g^{(1)})
\end{tikzcd} \]
Since the metrics $g$ and $g^{(1)}$ are left invariant, the maps $D_xL_{x^{-1}}$ and $D_{\pi_1(x)}L_{\pi_1(x^{-1})}$ are linear isometries. 
 Now the commuting diagram implies that 
  $D_x\pi_1:  (T_xG, g)\rightarrow  
  (T_{\pi_1(x)}S_1, g^{(1)})$  and  
  $D_e\pi_1:  (T_e G, g) \rightarrow 
 (T_e S_1,  g^{(1)})$ have the same operator norm.
 
\end{proof}

  Let  $p=(x_1, x_2, t)$ and $q=(y_1, y_2, s)$  with $t\le s$.  
  Let $\beta:[0, l]\rightarrow G$ be the arclength parametrization of  a length minimizing geodesic  from $q$ to $p$.  Let 
  $$h_+:=\max\{h(x)|x\in \text{im}(\beta)\}$$
  and 
  $$h_-:=\min\{h(x)|x\in \text{im}(\beta)\}. $$
   Set $D_+=t_{x_1, y_1}-h_+$ and 
  $D_-=h_--t_{x_2, y_2}$.

\begin{lemma}
\label{claim}
  There is a constant $C_0\ge 0$  independent of the points $p, q$ such that 
  $\max\{D_+, D_-\} \le C_0$. 
\end{lemma}

  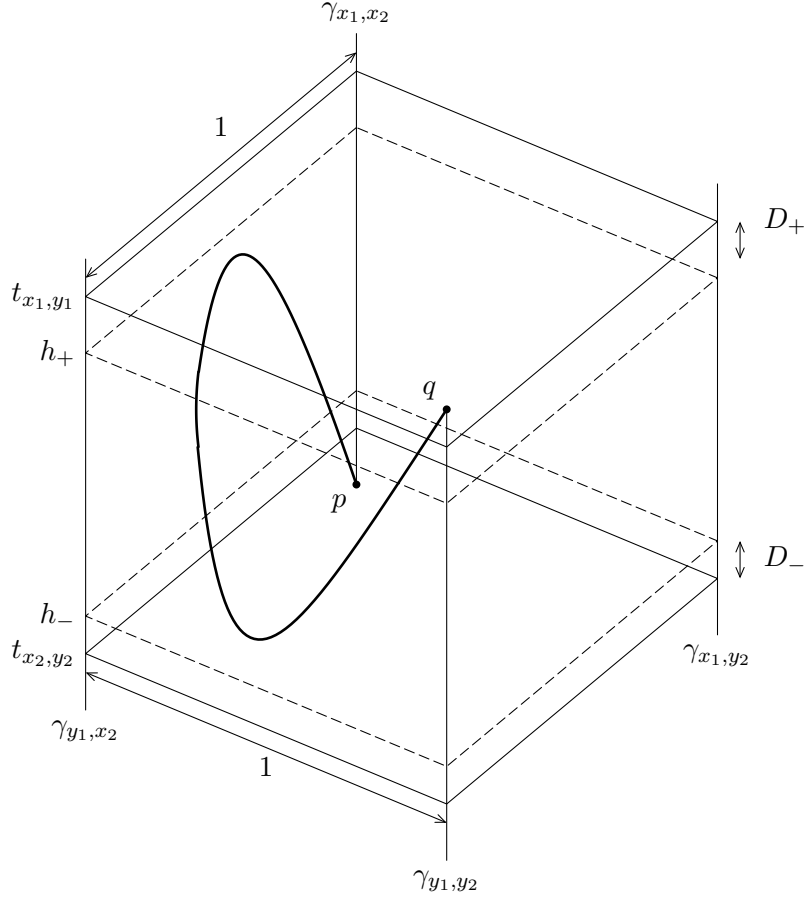
\begin{figure}
\begin{tikzpicture}[line cap=round,line join=round,>=angle 45,x=0.6cm,y=0.5cm]
\clip(-9,-15) rectangle (10,9);
\draw (-6,2)-- (-6,-10) node[below] {$\gamma_{y_1, x_2}$};

\draw (8,4)-- (8,-8) node[below] {$\gamma_{x_1, y_2}$};
\draw (0,-4) -- (0,8) node[above] {$\gamma_{x_1, x_2}$};
\draw (-6,1) -- (2,-3) -- (8,3) -- (0,7) -- cycle ;
\draw (2,-2) -- (2,-14) node[below] {$\gamma_{y_1, y_2}$};

\draw [color=black, line width=1pt,domain=0:6, samples = 50, variable = \t] plot({5*exp(-0.2*\t)-5},{-4+0.5*\t*(7-\t)});

\draw [shift={(0,-1.5)}, dash pattern = on 4pt off 2pt] (0,0)-- (-6,-6)-- (2,-10)-- (8,-4) -- cycle;

\draw [shift={(0,-2.5)}] (0,0)-- (-6,-6)-- (2,-10)-- (8,-4) -- cycle;

\draw [shift={(0,5.5)}, dash pattern = on 4pt off 2pt] (0,0)-- (-6,-6)-- (2,-10)-- (8,-4) -- cycle;

\draw [color=black, line width=1pt,domain=0:6.7, samples = 50, variable = \t] plot({6.37*exp(-0.3*\t)-4.37},{-2-0.5*\t*(7-\t)});

\draw [color=black, line width=1pt,domain=-1:1, samples = 50, variable = \t] plot({-3.55+0.05*\t*\t},{-2+\t)});

\draw (-6,1) node[anchor=east] {$t_{x_1,y_1}$};
\draw (-6,-0.5) node[anchor=east] {$h_+$};
\draw (-6,-7.5) node[anchor=east] {$h_-$};
\draw (-6,-8.5) node[anchor=east] {$t_{x_2,y_2}$};

\fill (0,-4) circle(1.5pt) node[anchor=north east]{$p$};
\fill (2,-2) circle(1.5pt) node[anchor=south east]{$q$};

\draw [<->] (-6,1.5) -- (0,7.5);
\draw (-3,5) node[anchor=south] {$1$};

\draw [<->] (-6,-9) -- (2,-13);
\draw (-2,-11) node[anchor=north] {$1$};


\draw [<->] (8.5,-6.5) -- (8.5,-5.5);
\draw (9.5,-6) node[anchor=center] {$D_-$};

\draw [<->] (8.5,2) -- (8.5,3);
\draw (9.5,3) node[anchor=center] {$D_+$};

\end{tikzpicture}
  \caption{Geodesic segment $\beta$ between $p$ and $q$ in coordinate view.}
  \end{figure}

  \begin{proof}[Proof of  Theorem~\ref{soldistancetheorem}  assuming  Lemma~\ref{claim}]
    We use the fact that the 
   minimal  distance between $N_1\times N_2\times \{t\}$ and $N_1\times N_2\times \{t'\}$  is $|t-t'|$.  
      Since $\ell_0$  is the length of a  curve between $q$ and $p$,  we have $d(p,q)\le \ell_0$.  We shall show that the reverse inequality holds up to an additive constant by using the expression for $\tilde \rho$.  Here we only write down the details for the case  $t_{x_2, y_2}\le  \min\{t, s\}$,    $t_{x_1, y_1}\ge \max\{t, s\}$
        as the other cases are similar.   First assume $\beta$  reaches height $h_+$  before it reaches height $h_-$.  
  From $q$ the curve  $\beta$ first reaches the height $h_+$,  so  this subcurve has length at least  $h_+-h(q)$.  Then $\beta$ goes down to the height $h_-$, so the length of this portion of $\beta$ is  at least $h_+-h_-$.   Finally $\beta$ goes up and reaches the height $h(p)$, so the length of this portion of $\beta$ is at least $h(p)-h_-$.    Hence the length of $\beta$ is   at  least 
  \begin{align*}
  &(h_+-h(q))+(h_+-h_-)
  +(h(p)-h_-)\\
  &=2h_+-2h_-  -(h(q)-h(p))\\
  &=2t_{x_1, y_1}-2D_+- 2t_{x_2, y_2}-2D_--(h(q)-h(p))\\
  &=\tilde \rho(p,q)-2-2D_+-2D_-\\
  &\ge \ell_0-4-2D_+-2D_-\\
  &\ge \ell_0-4-4C_0.
  \end{align*}
  
  If $\beta$  reaches height $h_-$  before it reaches height $h_+$, then a similar argument shows that its length is  greater than the quantity above. 
  \end{proof}
  
  \begin{proof}[Proof of Lemma~\ref{claim}]
  
  We will only consider the case  $D_+\ge D_-$   and show that $D_+\le C_0$. The case $D_-\ge D_+$ can be  similarly handled by considering $\pi_2$ instead of $\pi_1$. 
  We may assume that
  \begin{equation}\label{eq:D+islarge}
  D_+\ge 20\max\{C, 1\}
  \end{equation}
   with $C$ the constant from  \eqref{hausdorff}, otherwise we are done.   
   
  We consider three cases depending on
   the value of  $h_+-D_+$.
   
   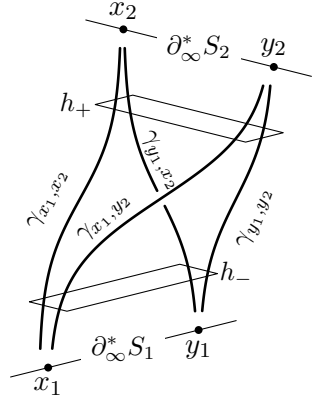
\begin{figure}
   
\begin{tikzpicture}[line cap=round,line join=round,>=angle 45,x=0.5cm,y=0.5cm]
\clip(-2,-6) rectangle (8,6);

\draw [color=black, line width=1pt,domain=-1.5:1.5, samples = 200, variable = \t] plot({5-tanh(\t)},{-2*\t});
\draw [color=black, line width=1pt,domain=-1.5:2, samples = 200, variable = \t] plot({5-(3*tanh(\t)+2)},{-2*\t});
\draw [color=black, line width=1pt,domain=-2:-0.1, samples = 200, variable = \t] plot({5-(2-tanh(\t))},{-2*\t});
\draw [color=black, line width=1pt,domain=0.1:1.5, samples = 200, variable = \t] plot({5-(2-tanh(\t))},{-2*\t});
\draw [color=black, line width=1pt,domain=-2:2, samples = 200, variable = \t] plot({5-(1.1*tanh(\t)+4.15)},{-2*\t});

\draw (-1,-4.75) -- (5,-3.25);
\draw (1,4.75) -- (7,3.25);

\fill (0,-4.5) circle(1.5pt) node[below] {$x_1$};
\fill (4,-3.5) circle(1.5pt) node[below] {$y_1$};
\draw (2,-4) node[stuff_fill,anchor=center] {\small $\partial_\infty^\ast S_1$};

\fill (2,4.5) circle(1.5pt) node[above] {$x_2$};
\fill (6,3.5) circle(1.5pt) node[above] {$y_2$};
\draw (4,4) node[stuff_fill,anchor=center] {\small $\partial_\infty^\ast S_2$};

\draw [shift={(1.25,2.5)}] (0,0) -- (1,0.25) -- (5,-0.75) -- (4,-1) -- cycle;

\draw [shift={(-0.5,-2.75)}] (0,0) -- (1,-0.25) -- (5,0.75) -- (4,1) -- cycle;

\draw (0,0) node[anchor=center, rotate=65] {\small $\gamma_{x_1,x_2}$};

\draw (1.5,-0.5) node[anchor=center, rotate=45] {\small $\gamma_{x_1,y_2}$};

\draw (5.5,-0.5) node[anchor=center, rotate=70] {\small $\gamma_{y_1,y_2}$};

\draw (3,1) node[anchor=center, rotate=-60] {\small $\gamma_{y_1,x_2}$};

\draw (5,-2) node[anchor=center] {\small $h_-$};

\draw (0.75,2.5) node[anchor=center] {\small $h_+$};

\end{tikzpicture}

   \caption{Metric view of the four vertical geodesics involved.}
   \end{figure}
  
 \begin{description}

 \item[{Case I}]    $h_+-D_+> h(q)$.

   We will divide the curve $\beta$ into several subcurves.   
  Let  $t_1\in [0,l]$ be the first  $t$ such that $h(\beta(t))\ge h_+-D_+$.  
    Let $t_2\in [0,l]$ be the last   $t$ such that $h(\beta(t))\ge h_+-D_+$.  Also let $l_0\in [0,l]$ be such that $h(\beta(l_0))=h_-$.  (There may be more than one such $l_0$; we just pick one.)
    
    \begin{enumerate}

  \item  Subcurve $\beta_1$:  
       Let $\beta_1=\beta|_{[0, t_1]}$.
      Set  $\ell_1=h_+-D_+-h(q)$  if $l_0\notin [0, t_1]$ and $\ell_1=(h(q)-h_-)+(h_+-D_+-h_-)$  if  $l_0\in [0, t_1]$ .  
        By considering the height change we see that $\ell(\beta_1)\ge \ell_1$.   Remember that $h(p) \le h(q)$ by assumption.
    Write $\beta(t_1)=(a_1, a_2,  h_+-D_+)$  with $a_i\in N_i$.

   \item Subcurve $\beta_2$:  $\beta_2=\beta|_{[t_1, t_2]}$. Set $\ell_2=2D_+$  if $l_0\notin [t_1, t_2]$ and $\ell_2=2(h_+-h_-)$  if $l_0\in [t_1, t_2]$ 
    Again, by considering the height change we obtain    
     $\ell(\beta_2)\ge \ell_2$.    
     
     Write $\beta(t_2)=(b_1, b_2,  h_+-D_+)$ with $b_i\in N_i$. 
   
   
   \item Subcurve $\beta_3$:  $\beta_3=\beta|_{[t_2, l]}$.
   Set  $\ell_3=(h_+-D_+)-h(p)$ if 
       $l_0\notin [t_2, l]$ and $\ell_3=((h_+-D_+)-h_-)+(h(p)-h_-)$    
       if   $l_0\in [t_2, l]$.  
       As above  we have $\ell(\beta_3)\ge \ell_3$.    
   
    \end{enumerate}
   We observe that 
   $\sum \ell_j\ge 2(h_+-h_-)-(h(q)-h(p))\ge   \ell_0-4-2D_+-2D_-\ge \ell_0-4-4D_+$.  
   (Recall that $\ell_0$ is the length of $\alpha_0$.)
   
    {\color{black}
   \begin{figure}
   
\begin{tikzpicture}[line cap=round,line join=round,>=angle 45,x=0.5cm,y=0.5cm]
\clip(-3,-3) rectangle (10,12);
\draw [->] (0,-2) -- (0,10);

\draw [color=black, line width=1pt,domain=1.15:1.66, samples = 500, variable = \t] plot({1.5*\t-0.05*\t*\t},{3+4*sin(((0.3*\t)+0.1*\t*\t) r)});

\draw [color=black, line width=1pt,domain=1.9:3.4, samples = 500, variable = \t] plot({1.5*\t-0.05*\t*\t},{3+4*sin(((0.3*\t)+0.1*\t*\t) r)});

\draw [color=black, line width=1pt,domain=3.65:6.45, samples = 500, variable = \t] plot({1.5*\t-0.05*\t*\t},{3+4*sin(((0.3*\t)+0.1*\t*\t) r)});

\draw [shift={(0,7)}, dash pattern = on 2pt off 2pt] (8,0) --(0,0) node[left]{$h_+$};
\draw [shift={(0,8)}, dash pattern = on 2pt off 2pt] (8,0) --(0,0) node[left]{$t_{x_1,y_1}$};
\draw [shift={(0,6)}, dash pattern = on 2pt off 2pt] (8,0) --(0,0) node[left]{$h_+-D_+$};
\draw [shift={(0,4.55)}, dash pattern = on 2pt off 2pt] (1.4,0) --(0,0) node[left]{$t= h(q)$};
\draw [shift={(0,2.55)}, dash pattern = on 2pt off 2pt] (7.5,0) --(0,0) node[left]{$s =h(p)$};
\draw [shift={(0,-1)}, dash pattern = on 2pt off 2pt] (6.8,0) --(0,0) node[left]{$h_-$};

\fill (1.4,4.55) circle(1.5pt) node[anchor= north west] {$\beta(0)$};
\fill (7.6,2.55) circle(1.5pt) node[anchor= south] {$\beta(l)$};
\fill (2.5,6) circle(1.5pt) node[anchor= south east] {$\beta(t_1)$};
\fill (4.7,6) circle(1.5pt) node[anchor= south west] {$\beta(t_2)$};
\fill (6.7,-1) circle(1.5pt) node[anchor= north] {$\beta(l_0)$};

\draw (2.5,4.8) node[anchor=center]{$\beta_1$};
\draw (4,7.5) node[anchor=center]{$\beta_2$};
\draw (6,4.5) node[anchor=center]{$\beta_3$};

\draw [<->] (9,2.55) -- (9,6);
\draw (9.5,4.2) node[anchor=center] {$\ell_1$};

\end{tikzpicture}
\caption{Case I with $l_0 \notin [0,t_2]$, image of $\beta$ through the projection $\pi_1$.}
   \end{figure}
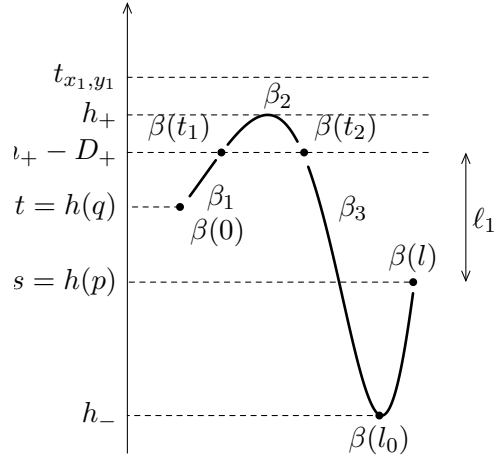

We now claim that

 \begin{equation}\label{d1-is-large}
 d^{(1)}((x_1, h_+-D_+), (y_1, h_+-D_+))\ge \frac{39}{10}D_+.
 \end{equation}
 Indeed, by the triangle inequality and the fact that the $d^{(1)}$-geodesic $\gamma_\cap$ between $(x_1, h_+-D_+)$ and $(y_1, h_+-D_+)$ lies in a $C$-neighborhood of the path $\gamma_{\sqcap}$ defined as a concatenation of vertical geodesics and length-minimizing segment in the horosphere of height $t_{x_1,y_1}$ between the same points (See Figure \ref{fig:d1-is-large}), we have that $ d^{(1)}((x_1, h_+-D_+), (y_1, h_+-D_+))\ge 2 (2D_+ - C)$ so that
 \begin{align*}
 d^{(1)}((x_1, h_+-D_+), (y_1, h_+-D_+)) & \ge 2 \left( 2D_+ - \frac{1}{20} D_+ \right) \qquad \text{ by \eqref{eq:D+islarge} } \\
  & = 4D_+ - \frac{1}{10} D_+ = \frac{39}{10} D_+.
 \end{align*}
   }
   
   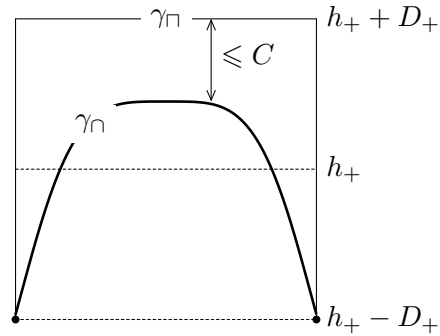
\begin{figure}
   \begin{tikzpicture}[line cap=round,line join=round,>=angle 45,x=1.0cm,y=1.0cm]
\clip(-3.7,-2) rectangle (3.7,3.5);
\draw (-2,-1)-- (-2,3);
\draw (-2,3)-- (2,3) node[right]{$h_++D_+$};
\draw (2,-1)-- (2,3);
\draw [dash pattern=on 1pt off 1pt] (-2,1)-- (2,1) node[right]{$h_+$} ;
\draw [dash pattern=on 1pt off 1pt] (-2,-1)-- (2,-1) node[right]{$h_+ - D_+$} ;
\fill [color=black] (-2,-1) circle (1.5pt);
\fill [color=black] (2,-1) circle (1.5pt);
\draw [color=black, line width=1pt,domain=-2:2, samples = 500, variable = \t] plot({\t},{2*cos(\t*\t/2 r)-0.1});

\draw (0,3) node[stuff_fill, anchor=center] {$\gamma_\sqcap$};

\draw (-1,1.6) node[stuff_fill, anchor=center] {$\gamma_\cap$};

\draw [<->] (0.6,3)-- (0.6,1.9);

\draw (0.6,2.5) node[anchor=west] {$\leqslant C$};

\end{tikzpicture}
\caption{Proof of inequality \eqref{d1-is-large}. The picture is in $(S_1, d^{(1)})$.}
\label{fig:d1-is-large}
   \end{figure}

     Since 
     \[ \frac{39}{10} = \frac{7}{2} + \frac{1}{5} + \frac{1}{5}, \]
       by the  triangle inequality one of the following holds:
\begin{align}
d^{(1)}((y_1, h_+-D_+), (a_1, h_+-D_+)) & \ge D_+/5; \tag{Condition I.1} \label{condI1} \\
d^{(1)}((a_1, h_+-D_+), (b_1, h_+-D_+)) & \ge \frac{7}{2}D_+; \tag{Condition I.2} \label{condI2} \\
d^{(1)}((b_1, h_+-D_+), (x_1, h_+-D_+))& \ge D_+/5 \tag{Condition I.3}. \label{condI3} 
\end{align}

  Since $\beta $ is a  minimizing geodesic between $q$ and $p$ and $\ell_0$ is the length of a path between $p$ and $q$ we have 
  $\ell(\beta)\le \ell_0$.    
  On the other hand,  Condition  (I.1) or (I.2) or (I.3) holds.

  We first   assume \eqref{condI1} holds.   Let $y$ be the point on the vertical geodesic $\gamma_{y_1, y_2}$ through $q$ with $h(y)=h_+-D_+$, and $qy$ the segment of $\gamma_{y_1, y_2}$  between $q$ and $y$.  
    Then  the curve $\alpha_1:=\beta_1\cup qy$ is a path joining $y$ and $\beta(t_1)$ that lies below the horosphere $h=h_+-D_+$.  
(Recall that $\beta$ starts from $q$.)    
    
    To simplify notation denote
      $q'=\pi_1(y)=(y_1, h_+-D_+)$ and $p'=\pi_1(\beta(t_1))=(a_1, h_+-D_+)$. Then \eqref{condI1} simply  says $d^{(1)}(p',q')\ge D_+/5$.    
      \begin{description}
      
     \item[Subcase I.1.a]   $H(\alpha_1)\le d^{(1)}(p',q')$. 
          We use  
        Lemma~\ref{comparison-horosphere}   
            to conclude  $\ell(\pi_1\circ  \alpha_1)\ge  2^{\frac{d^{(1)}(p',q')-C-2}{2\delta}} -C$.    
             By Lemma~\ref{projection}  the map $\pi_1$ is $L$-Lipschitz for some $L>0$, and so we have
               \[ \ell(\alpha_1)\ge  \ell(\pi_1\circ \alpha_1)/L\ge \frac{1}{L} 2^{\frac{d^{(1)}(p',q')-C-2}{2\delta}} -C/L.\]  
               As  $d(q,y)\le H(\alpha_1)\le   d^{(1)}(p',q')$, we have 
                    $\ell(\beta_1)=\ell(\alpha_1)-d(q,y)\ge \frac{1}{L} 2^{\frac{d^{(1)}(p',q')-C-2}{2\delta}} -C/L-d^{(1)}(p',q')$.  
                  Note that $\ell_1\le 2 H(\beta_1)=2 H(\alpha_1)\le 2 d^{(1)}(p',q')$.   Now we  have 
              \begin{align*}
              \ell_0 &\ge 
              \ell(\beta)\\
             & =\ell(\beta_1)+\ell(\beta_2)+\ell(\beta_3)\\
              &\ge \ell(\beta_1)  +\ell_2+\ell_3\\
              &=(\ell(\beta_1) -\ell_1)+(\ell_1 +\ell_2+\ell_3)\\
              &\ge  ( \frac{1}{L} 2^{\frac{d^{(1)}(p',q')-C-2}{2\delta}} -C/L-d^{(1)}(p',q')-2d^{(1)}(p',q'))+(\ell_0-4-4D_+)\\
              &\ge \ell_0-4-23 d^{(1)}(p',q')+\frac{1}{L} 2^{\frac{d^{(1)}(p',q')-C-2}{2\delta}} -C/L, 
              \end{align*}
              where   for the last inequality we used $d^{(1)}(p',q')\ge D_+/5$.
               Now it is clear that $d^{(1)}(p',q')$ and so $D_+$ is bounded above by a constant depending only on $L$, $C$, and $\delta$.

            \item[Subcase I.1.b] $H(\alpha_1)> d^{(1)}(p',q')$.  
    Then    Lemma~\ref{path-outside}  (2)   applied to $\pi_1\circ \alpha_1$   implies there is a subcurve $\tilde\beta_1$ of $\beta_1$ that is either   a segment or the union of two segments  of $\beta_1$   such that the height change (of the end points of the segments in $\tilde \beta_1$) is at most $4d^{(1)}(p',  q')$  and the length of $\pi_1\circ \tilde\beta_1$ satisfies
   $$\ell(\pi_1\circ \tilde \beta_1)\ge   2^{\frac{d^{(1)}(p',q')-C-2}{2\delta}}-C-3d^{(1)}(p',q').$$
    On the other hand, the map $\pi_1$ is $L$-Lipschitz  so  we have $\ell(\tilde \beta_1)\ge  \ell(\pi_1\circ \tilde\beta_1)/L$.  
    
    Let $\tilde{\tilde{\beta}}_1$ be  the complement of $\tilde \beta_1$ in $\beta_1$.  Since  the height change of $\beta_1$ is at least $\ell_1$ and the height change of $\tilde \beta_1$ is at most $4d^{(1)}(p',q')$, we see the length of $\tilde{\tilde{\beta}}_1$  satisfies $\ell(\tilde{\tilde{\beta}}_1)\ge \ell_1-4d^{(1)}(p',q')$.   Together with the estimate of the length of $\tilde\beta_1$ from the above paragraph we get $\ell(\beta_1)\ge \ell_1+\frac{1}{L}2^{\frac{d^{(1)}(p',q')-C-2}{2\delta}}-C/L-(3/L+4)d^{(1)}(p',q')$.     
    It follows that 
    \begin{align*}
    \ell_0  & \ge  \ell(\beta)\\
    & \ge \ell_1+\ell_2+\ell_3+\frac{1}{L}2^{\frac{d^{(1)}(p',q')-C-2}{2\delta}}-C/L-(3/L+4)d^{(1)}(p',q')\\
    &\ge \ell_0-4-4D_++  \frac{1}{L}2^{\frac{d^{(1)}(p',q')-C-2}{2\delta}}-C/L-(3/L+4)d^{(1)}(p',q')\\
    &\ge \ell_0-4+  \frac{1}{L}2^{\frac{d^{(1)}(p',q')-C-2}{2\delta}}-C/L-(3/L+24)d^{(1)}(p',q'),
    \end{align*}
      where   again for the last inequality we used $d^{(1)}(p',q')\ge D_+/5$.
      Now it is clear that $d^{(1)}(p',q')$ and so $D_+$ is bounded above by a constant depending only on $L$, $C$, and $\delta$. 
    
      \end{description}
      

       This finishes the proof  of the Claim  in Case I, Condition \eqref{condI1}.
       
        Now assume \eqref{condI2} holds.  This condition implies $d^{(1)}((a_1, h_+),  
         (b_1, h_+))>D_+$.     Since 
         $\alpha_2:=\gamma_{a_1, a_2}|_{[h_+-D_+, h_+]}\cup \beta_2\cup \gamma_{b_1, b_2}|_{[h_+-D_+, h_+]}$  lies below the height $h_+$,
          an argument similar to the case of \eqref{condI1} finishes the proof.   
      
    Finally we assume \eqref{condI3} holds.  In this case the curve  $\alpha_3:=\gamma_{x_1, x_2}|_{[h(p), h_+-D_+]}\cup \beta_3$ lies below the height
     $h_+-D_+$ and we repeat the above argument.

    
    \item
 [Case II]    $h(p)<h_+-D_+\le  h(q)$.

    In this case 
 we  divide the curve $\beta$ into  two  subcurves.   
    Let $t_2\in [0,l]$ be the last   $t$ such that $h(\beta(t))\ge h_+-D_+$.  Also let $l_0\in [0,l]$ be such that $h(\beta(l_0))=h_-$.  
     
    \begin{enumerate}

     \item  Subcurve $\beta_2$:  $\beta_2=\beta|_{[0, t_2]}$. Set $\ell_2=(h_+-h(q))+(h_+-(h_+-D_+))$  if $l_0\notin [0, t_2]$ and 
     $\ell_2=(h_+-h(q))+(h_+-h_-)+(h_+-D_+-h_-)$  if $l_0\in [0, t_2]$ .  By considering the height change we obtain    
     $\ell(\beta_2)\ge \ell_2$.    
     Write $\beta(t_2)=(b_1, b_2,  h_+-D_+)$ with $b_i\in N_i$. 
   
   
   \item  Subcurve $\beta_3$:  $\beta_3=\beta|_{[t_2, l]}$.
   Set  $\ell_3=(h_+-D_+)-h(p)$ if 
       $l_0\notin [t_2, l]$ and $\ell_3=((h_+-D_+)-h_-)+(h(p)-h_-)$    
       if   $l_0\in [t_2, l]$.  As above  we have $\ell(\beta_3)\ge \ell_3$.    
     \end{enumerate}
     \begin{figure}
     
\begin{tikzpicture}[line cap=round,line join=round,>=angle 45,x=0.5cm,y=0.5cm]
\clip(-3,0) rectangle (10,12);
\draw [->] (0,1) -- (0,10);

\draw [color=black, line width=1pt,domain=1.15:1.66, samples = 500, variable = \t] plot({9-(1.5*\t-0.05*\t*\t)},{3+4*sin(((0.3*\t)+0.1*\t*\t) r)-0.4*ln(\t-1)*ln(\t-1)});

\draw [color=black, line width=1pt,domain=1.9:3.65, samples = 500, variable = \t] plot({9-(1.5*\t-0.05*\t*\t)},{3+4*sin(((0.3*\t)+0.1*\t*\t) r)});

\draw [color=black, line width=1pt,domain=3.65:6.2, samples = 500, variable = \t] plot({9-(1.5*\t-0.05*\t*\t)},{3+4*sin(((0.3*\t)+0.1*\t*\t) r)+0.35*(\t-3/65)*(\t-3.65)});

\draw [color=black, line width=1pt,domain=6.2:6.43, samples = 500, variable = \t] plot({9-(1.5*\t-0.05*\t*\t)},{3+4*sin(((0.3*\t)+0.1*\t*\t) r)+0.35*(\t-3/65)*(\t-3.65)});

\draw [shift={(0,6)}, dash pattern = on 2pt off 2pt] (6,0) --(0,0) node[left]{$h_+-D_+$};
\draw [shift={(0,3)}, dash pattern = on 2pt off 2pt] (7.35,0) --(0,0) node[left]{$s= h(p)$};
\draw [shift={(0,8.55)}, dash pattern = on 2pt off 2pt] (1.5,0) --(0,0) node[left]{$t =h(q)$};
\draw [shift={(0,2.2)}, dash pattern = on 2pt off 2pt] (2.2,0) --(0,0) node[left]{$h_-$};

\fill (7.35,3) circle(1.5pt) node[anchor= north west] {$\beta(l)$};
\fill (1.4,8.55) circle(1.5pt) node[anchor= south] {$\beta(0)$};
\fill (6.5,6) circle(1.5pt) node[anchor= south west] {$\beta(t_2)$};
\fill (2.55,2.2) circle(1.5pt) node[anchor= north] {$\beta(l_0)$};

\draw (1,7.5) node[anchor=center]{$\beta_2$};
\draw (6,4.5) node[anchor=center]{$\beta_3$};


\end{tikzpicture}
     
     \caption{$\pi_1$-projection of $\beta$ in Case II when $l_0 \in [0,t_2]$.}
     \end{figure}
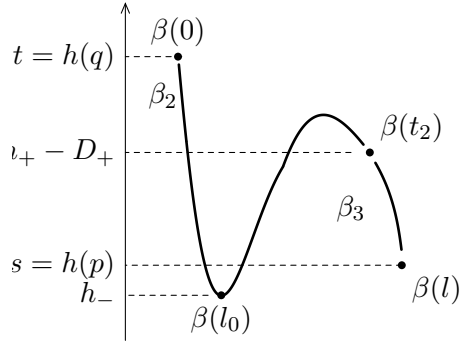
   
   We observe that   $\ell_2+\ell_3\ge 2(h_+-h_-)-(h(q)-h(p))\ge \ell_0-4-4D_+$  as before.  
   
    As before
     $d^{(1)}((x_1, h_+-D_+), (y_1, h_+-D_+))\ge   \frac{39}{10}D_+$ and so by     the  triangle inequality one of the following holds:

     \begin{align}
       d^{(1)}((y_1, h_+-D_+), (b_1, h_+-D_+))\ge \frac{37}{10}D_+; \label{condII1} \tag{Condition II.1} \\ 
   d^{(1)}((b_1, h_+-D_+), (x_1, h_+-D_+))\ge D_+/5. \label{condII2} \tag{Condition II.2}
  \end{align}
  
    First  assume \eqref{condII1} holds.  This implies   $d^{(1)}((y_1, h_+), (b_1, h_+))>D_+$.   The curve 
     $\alpha_2:=\gamma_{b_1, b_2}|_{[h_+-D_+, h_+]}\cup \beta_2\cup \gamma_{y_1, y_2}|_{[h(q), h_+]}$   lies below the height $h_+$ 
       and we can repeat the  argument in Case I to finish the proof. 
  
  Finally we assume \eqref{condII2} holds. In this case the curve $\alpha_3:=\gamma_{x_1, x_2}|_{[h(p), h_+-D_+]}\cup \beta_3$ lies below the height $h_+-D_+$ 
   and we can repeat the  argument in Case I to finish the proof.

  \item[{Case III}]  $h_+-D_+\le h(p)$.  As before we have $d^{(1)}((x_1, h_+-D_+), (y_1, h_+-D_+))\ge   \frac{39}{10}D_+$, which implies
   $d^{(1)}((x_1, h_+), (y_1, h_+))>D_+$.  In this case the curve $\alpha:=\gamma_{x_1, x_2}|_{[h(p), h_+]}\cup \beta\cup \gamma_{y_1, y_2}|_{[h(q), h_+]}$ lies below
    the height $h=h_+$ and we can apply the previous argument to conclude. \qedhere
  \end{description}
  
  \end{proof}

\section{Reformulating former results}
\label{sec:reformulation}

\subsection{A special case of the pointed sphere conjecture}
\label{subsec:spec-case-ps}

We shall refer here to the pointed sphere conjecture of Cornulier recorded in \cite[Conjecture 19.104]{Cornulier:qihlc}.


{\color{black} 
By first stratum of a Carnot algebra with Carnot derivation $D$, we mean the eigenspace $\ker(D-1)$, which by assumption Lie generates the Carnot algebra, see  \cite{ledonne_primer}. The higher strata are the subspaces $\ker(D-i)$ for $i \geqslant 2$; the Lie algebra is a direct sum of its strata. (One also encounters the term {\em layer} in the literature.)}

Let  $N$ be a Carnot group with Lie algebra $\mathfrak n$ and first stratum $V_1$. We say that $N$ (or equivalently $\mathfrak n$) has {reducible first stratum} if there is a nontrivial subspace $W$ of $V_1$ 
such that for 
every strata-preserving automorphism $\phi$ of $\mathfrak n$ one has $\phi(W)=W$.
Such a notion has been studied in  \cite{XieReducible}, however, the reader should not mistake it with the notion of reducibility, from the same paper.

 \begin{proposition}
 \label{prop:pointed-sphere-red-first-stratum}
Let $S=N\rtimes \mathbb R$ be a Heintze group of Carnot type. 
Assume that $N$ has reducible first stratum.
Then, the pointed sphere conjecture holds for $S$.
Namely, every quasi-symmetric self-homeomorphism of $\partial S$ fixes the focal point in $\partial S$.
\end{proposition}

\begin{figure}
\includegraphics[scale=0.1]{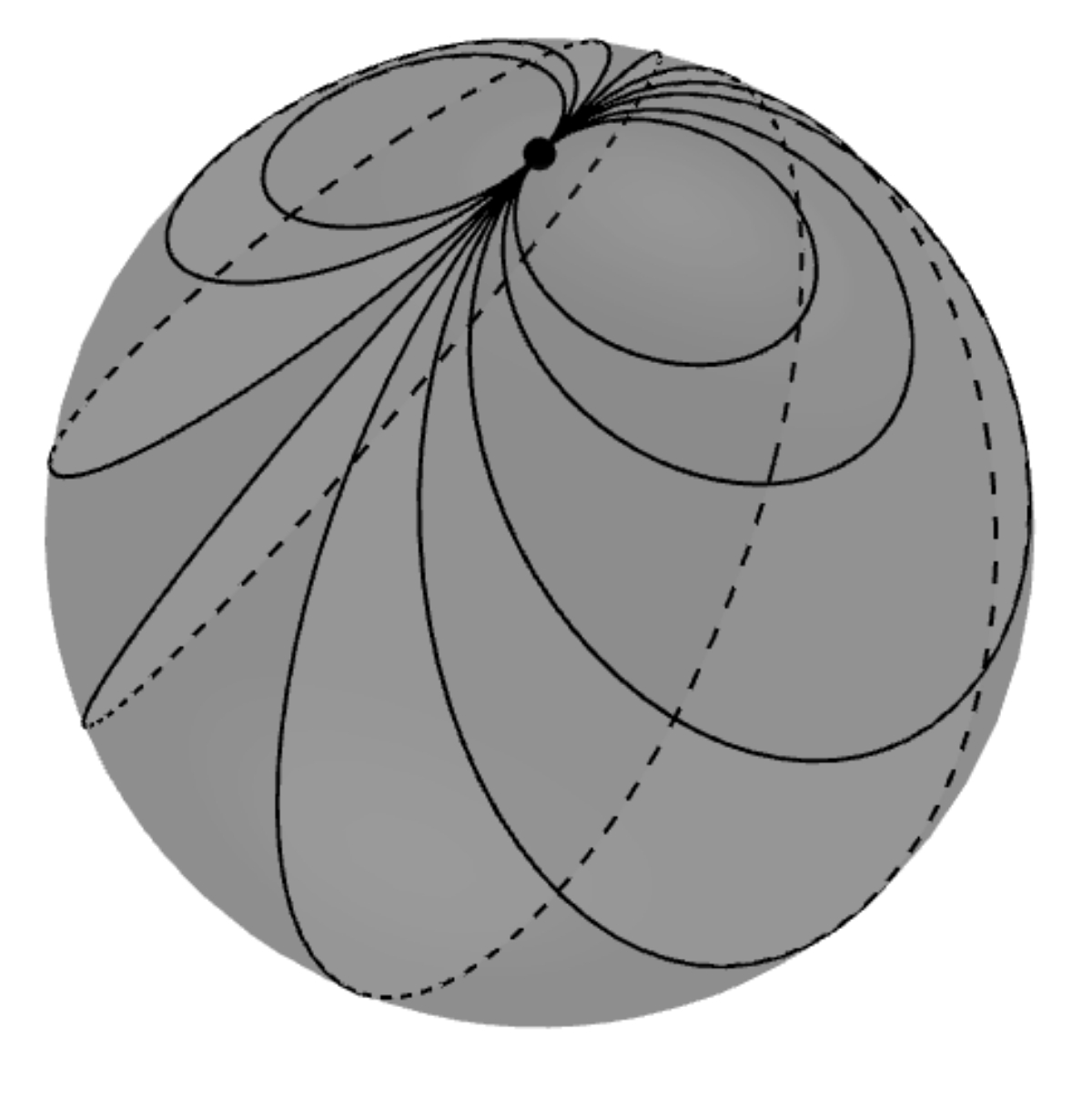}
\caption{Invariant foliation on the boundary at infinity.}
\label{fig:invariant-foliation}
\end{figure}

\begin{proof}
The argument that we shall follow is similar to the one in \cite{LeDonneXie} and it is based on \cite{XieReducible}, such a  principle goes back to \cite[before Corollaire 6.9]{PansuDimConf}. 
Namely, we shall prove that the focal point in $\partial S$ is fixed by proving that a special foliation in $\partial S$ is preserved (See Figure~\ref{fig:invariant-foliation}).

Let $\omega$ be the focal point in $\partial S$, so that  $\partial S = N\cup \{\omega\}$.
Let $F\colon \partial S \to \partial S$ be a quasi-symmetric homeomorphism. We need to prove that $F(\omega)=\omega$. Let us assume that this is not the case.

Since $N$ is assumed to have reducible first stratum $V_1$, then there is a nontrivial subspace $W$ of $V_1$ that is fixed by (the differential of) every strata-preserving automorphism of $N$. Consequently, the nontrivial group $G$ generated by $\exp(W)$ is preserved by every strata-preserving automorphism of $N$.

The cosets of $G$ induce a singular foliation on $\partial S$. 
In particular, when we restrict to the sets 
$U_1:= N\setminus F^{-1}(\omega)$ and $U_2:=N\setminus F(\omega)$,
we have that   the leaves on $U_1$ (resp. on $U_2$) are exactly the left cosets $xG$, as $x\in N$, except for a leave, which is 
$F^{-1}(\omega)G \setminus {F^{-1}(\omega)}$ (resp. $F(\omega)G \setminus {F(\omega)}$). At this point we stress that in $\partial S$ while $F^{-1}(\omega)$ and $  F(\omega)$ are in the closure of just one of these leaves, the point $\omega$ is in the closure of every leave.

Let us restrict to the map 
$\hat F:=F|_{U_1} :U_1 \to U_2$.
Since $U_1$ and $U_2$ are open set of the Carnot group $N$, 
by Pansu’s differentiability theorem \cite{PansuCCqi}, 
 the map $\hat F$ is Pansu differentiable at almost every point and its Pansu differential 
  is a strata-preserving automorphism, which therefore preserves the proper subgroup $G$.
By the argument in   \cite[Proposition 3.4]{XieReducible} we have that $\hat F$ preserves 
the leaves of the foliation that we are considering.
Since we had a topological characterization of the point $\omega$ (and since the map $F$ is continuous), then we get to a contradiction unless $F$ fixes $\omega$.
\end{proof}

\subsection{Restatement of Le Donne-Xie's theorem}

\begin{theorem}[After Le Donne-Xie]
\label{ref-LDX}
Let $N$ be a Carnot group with reducible first stratum.
Let $S=N \rtimes \mathbb R$ be the Carnot-type Heintze group associated to $N$, and equip $S$ with {any left-invariant Riemannian distance}.
Then, every self-quasiisometry of $S$ is a rough isometry.
\end{theorem}

\begin{proof}
 Let $g_0$ be  a  left-invariant Riemannian metric on $S=N\rtimes \mathbb R$ such that the $N$ direction is perpendicular to $\mathbb R$, let $d_0$ be the associated 
        distance.   
   Then the  ideal boundary $\partial S$ can be identified with $N\cup \{\omega\}$, and  the Carnot metric on $N$  is a parabolic visual metric with respect to  $\omega$.   
   Now  
     let $g$ be  an arbitrary  left-invariant Riemannian metric on $S$ with $d$ the associated distance,  and 
   let  $\phi$ be a self quasiisometry of $(S, d)$.  Then $\phi$ is also a self quasiisometry of $(S, d_0)$. 
By Proposition~\ref{prop:pointed-sphere-red-first-stratum}, $\partial \phi: \partial S\rightarrow  \partial S$ fixes the focal point of $S$  and so induces a  self quasisymmetric map $\partial^\ast \phi$
    of $N$ (with the Carnot metric).  
    Then $\partial^\ast \phi$ is a biLipschitz homeomorphism of $N$ \cite[Theorem 1.2]{LeDonneXie}. However, a self quasiisometry  of  a  Gromov-hyperbolic space is a rough isometry if and only if the induced boundary  map is biLipschitz.     This follows from the results of Bonk-Schramm   \cite{BonkSchrammEmbeddingsCorrectedInSchrammsWork} (Theorems 7.4 and 8.2). For a direct proof in the case of parabolic visual metric see \cite[Lemma~5.1]{SX}.
     Hence $ \phi: (S, d_0)\rightarrow (S, d_0)$ is a rough isometry.  By  Theorem~\ref{main-heintze}   the identity map $\text{Id}: (S, d_0)\rightarrow (S, d)$ is a rough isometry. 
     It follows that $\phi: (S, d)\rightarrow (S, d)$ is also a rough isometry.  
\end{proof}

{\color{black}
\begin{remark}\label{rem:C-is-reducible}
Pansu defined the Carnot-type groups of class (C) in \cite[14.1]{PansuCCqi}.
At the Lie algebra level, the definition reads as follow: the Carnot-type group  $N\rtimes_D \mathbb R$ (where $N$ is different from $\mathbb R$) is of Pansu's class (C) if the centralizer of $\mathbb RD$ in the Lie algebra of derivations of $\mathfrak n$ is equal to $\mathbb RD$ itself.
Since the centralizer of $\mathbb RD$ Lie generates the strata-preserving automorphisms of $\mathfrak n$, and since the first stratum $V_1$ of a nilpotent Lie algebra has dimension greater than or equal $2$, the first stratum of a Lie algebra of class (C) has a vector $w_1$ generating a proper subspace that is invariant under the strata-preserving automorphisms (in fact, any nonzero vector in the first stratum is good for this).
It follows that the groups in Pansu's class (C) have reducible first stratum.
Pansu proved that among Carnot groups $N$ of class $2$ with $\dim V_1$ even, greater or equal to $10$ and $3 \leqslant \dim N - \dim V_1 \leqslant 2 \dim V_1 - 3$, the property of being of class (C) is {\em generic} in the sense of algebraic geometry \cite{PansuCCqi}. This implies that having reducible first stratum is also a generic property among these groups.
As mentionned in the Introduction, Theorem \ref{ref-LDX} for the groups in Pansu's class (C) is due to Pansu.
\end{remark}
}

\subsection{Restatement of Eskin-Fisher-Whyte's and Ferragut's  theorems}


\begin{theorem}[After Eskin-Fisher-Whyte]
\label{th:ref-EFW}
Let $G$ be the Lie group SOL.
Equip $G$ with {any left-invariant Riemannian distance}.
Then, every self-quasiisometry of $G$ is a rough isometry.
\end{theorem}

\begin{theorem}[After Ferragut]
\label{th:ref-F}
Let $G$ be a non-unimodular Sol-type group.
Equip $G$ with {any left-invariant Riemannian distance}.
Then, every self-quasiisometry of $G$ is a rough isometry.
\end{theorem}

\begin{proof}[Proof of Theorems~\ref{th:ref-EFW} and~\ref{th:ref-F}]
Start assuming that $G$ is a non-unimodular Sol-type group.
Equip $G$ with a horospherical product Riemannian metric $g_{0}$, that is, a metric for which $\mathfrak{n}_1 \perp \mathfrak{n}_2$.
Decompose $G = N_1 \times N_2 \times \mathbb R$ where the direction of the $\mathbb R$ factor is $g_0$-perpendicular to $N_1 \times N_2$. 
By \cite[Theorem 10.3.2]{FerragutThesis}, every quasiisometry $\Phi$ of $G$ is a bounded distance away from $(\Psi_1, \Psi_2, \operatorname{Id}_{\mathbb R})$, where $\Psi_i$ is bilipschitz with respect to the $D_i$-parabolic metric on $N_i$.
Using again \cite{SX} $(\Psi_1, \operatorname{Id}_{\mathbb  R})$ is a rough isometry of $S_1$ while $(\Psi_2, \operatorname{Id}_{\mathbb R})$ is a rough isometry of $S_2$. And then we conclude by Theorem~\ref{soldistancetheorem} that $\Phi$ is a rough isometry of $(G,g_0)$.
Then by Theorem~\ref{main-soltype}, $\Phi$ is a rough isometry of $G$ with respect to any left-invariant Riemannian distance.
Now, let $G$ be the three-dimensional Lie group SOL equipped with its standard metric written in coordinates $(n_1, n_2, t)$ as
\[ ds^2 = e^{-2t} dn_1^2 + e^{2t} dn_{2}^2 + dt^2, \]
and $\Phi$ is a quasiisometry of $G$, it follows from \cite{EFW2} that up to possibly composing $\Phi$ with an isometry, $\Phi$ is at bounded distance from a product map of the form above. The end of the argument is the same as before.
\end{proof}

\subsection{Restatement of Carrasco Piaggio's theorem}

\begin{theorem}[After Carrasco Piaggio]
Let $S$ be a Heintze group.
Assume that the real shadow of $S$ is not of Carnot type.
Equip $S$ with a left-invariant Riemannian metric.
Then, every self quasiisometry of $S$ is a rough isometry.
\end{theorem}

\begin{proof}
Let $g_0$ be the left-invariant Riemannian metric on $S$ which is simultaneoulsy isometric to a left-invariant metric $\tilde g_0$ on the real shadow $S_0$ \cite{AlekHMN}; denote by $\rho \colon S \to S_0$ any such isometry.
By the published version of \cite[Corollary 1.8]{CarrascoOrliczHeintze}, every self quasiisometry of $S_0$ is a rough isometry. Let $\phi$ be a self quasiisometry of $S$. Then $\rho \phi \rho^{-1}$ is a self quasiisometry of $S_0$, hence a rough isometry of $S_0$ with respect to $\tilde g_0$. It follows that $\phi$ is a rough isometry of $g_0$, and then of any left-invariant metric by Theorem~\ref{main-heintze}.
\end{proof}

\subsection{Restatement of Kleiner-M\"uller-Xie's theorems}

\begin{theorem}[After Kleiner, M\"uller and Xie]
Let $S$ be a Heintze group whose real shadow is of Carnot type. 
Assume that the nilradical $N=[S,S]$ is nonrigid in the sense of Ottazzi-Warhurst, and that $N$ is not $\mathbb R^d$ or a Heisenberg group.
Equip $S$ with a left-invariant Riemannian metric.
Then, every self quasiisometry of $S$ is a rough isometry.
\end{theorem}

\begin{proof}
The pointed sphere conjecture holds for these groups by \cite[Theorem 1.2]{KMX21} and global quasisymmetric homeomorphisms of the boundary minus the focal point are bilipschitz by \cite[Theorem 3.1]{KMX21}.
The mechanism of proof is then exactly the same as for Theorem~\ref{ref-LDX}.
\end{proof}

{\color{black}
\begin{theorem}[After Kleiner, M\"uller and Xie]
Let $S$ be a Heintze group whose real shadow is of Carnot type. 
Assume that the nilradical $N=[S,S]$ is the group of unipotent triangular real $n \times n$ matrices, $n\geqslant 4$.
Equip $S$ with a left-invariant Riemannian metric.
Then, every self quasiisometry of $S$ is a rough isometry.
\end{theorem}

\begin{proof}
Global quasisymmetric homeomorphisms of the boundary minus the focal point are bilipschitz by \cite[Theorem 1.3]{LDX22f}.
By \cite[Corollary 3.2]{LDX22f}, there is an automorphsim $\tau$ of $N$, such that possibly after composing with $\tau$, any local quasiconformal homeomorphism of the boundary locally preserve a coset foliation. The pointed sphere conjecture for $S$ can be deduced in the same way as we did in \S \ref{subsec:spec-case-ps}.
\end{proof}

}

\section{Limitations of the present work and questions left open}
\label{sec:lamplighter}

\subsection{Failure of the analogous property for Lamplighter groups}

The groups $L_m = \mathbb Z /m \mathbb Z \wr \mathbb Z$ for $m \geqslant 2$ share their asymptotic cones (namely, horospherical products of two $\mathbb R$-trees equipped with a preferred horofunction, see \cite[Section 9]{CornulierDimCone}) with that of the group SOL, and their large-scale geometries are in many respect comparable. 
However we will prove below that they fail to have their word metrics roughly similar through the identity map.

Consider the following infinite presentation of the group $L_m$:
\[ \langle a, t \mid a^m, [t^i at^{-i}, t^j a t^{-j}], i, j \in \mathbb Z \rangle\]
and the two finite generating sets
\begin{itemize}
\item 
The wreath product generating set $S_w = \{a,t \}$
\item
The automaton generating set $S_a = \{t, ta \}$.
\end{itemize}
We denote by $d_w$ and $d_a$ the word distances with respect to $S_w \cup S_w^{-1}$ and $S_a \cup S_a^{-1}$ respectively.
In the following, by ``color'' we mean an element of $\mathbb Z/m\mathbb Z$.
An element of $L_m$ is encoded by a lighting function $\mathbb Z \to \mathbb Z/m\mathbb Z$ together with the position of a cursor, with the following multiplication law:
Multiplying by $t$ on the right amounts to moving the cursor to the right, and
multiplying by $a$ on the right amounts to shifting color at the position of the cursor.
Let $n$ be a positive integer (think of it large enough).
Let us first consider the element $g \in L_m$ for which the cursor is located at $0 \in \mathbb Z$ and the bulb at position $n$ is lit with the color $1 \in \mathbb Z / m \mathbb Z$  (all the others being not lit). Note that $g = t^n a t^{-n}$. 
Since in both generating sets, the cursor moves at most by one unit at each multiplication by a generator, the distance from $1$ to $g$ in both word metrics must be at least $2n$. In fact, one computes that
\[ d_w (1, t^n a t^{-n}) = 2n+1 \]
while {
\[ t^n a t^{-n} = (ta)^n a (ta)^{-n} = (ta)^n  t^{-1}(ta) (ta)^{-n} = (ta)^n  t^{-1} (ta)^{-(n-1)} \]
hence
$d_a(1,t^n at^{-n}) = 2n$.}
It follows that if $d_a$ and $d_w$ were to be roughly similar through the identity, they should differ by a constant.
However, 
$d_w(1, (ta)^n) = 2n$
while $d_a(1,(ta)^n) = n$
as may be proved by counting the occurrences of $t$ and the number of bulbs lit in the final configuration.

\begin{remark}
The wreath product metric is easier to undersand, and there are explicit formulae for the word length for families of words in normal form. Taback and Cleary have investigated the geometry of the automata metric and the result above could be deduced from their paper \cite{ClearyTaback}.
\end{remark}

    \subsection{In search of a coarse notion}
    
    One of the main limitations of our present work is that the property that we identify, namely having all the left-invariant Riemannian metrics roughly similar, is not a coarse property.
    Indeed, Riemannian metrics play no special role among proper geodesic metrics as far as large-scale geometry is concerned.
    
    The search for a coarse notion leads to the following considerations.
Let $G$ be a locally compact, compactly generated group.
Denote by $\operatorname{Geom}(G)$ the collection of geometric actions, that is, pairs $(X, \alpha)$ where $X$ is proper geodesic and $\alpha \colon G \to \operatorname{Isom}(X)$ is continuous, proper and cocompact.
For every pair $\lbrace (X, \alpha), (Y, \beta) \rbrace$ in $\operatorname{Geom}(G)$, and for every pair of points $o_X \in X, o_Y \in Y$, the map $G._\alpha o_X \to G._\beta o_Y$ determined by the identity map of $G$ is a quasiisometry $X \to Y$. We call this map (to be considered only up to bounded distance) the $G$-{orbital map}. 
If $H < G$ is closed and co-compact, then $H$ is still compactly generated locally compact \cite[2.C.8(3)]{CornulierHarpeMetLCGroups}, and there is a natural map $\operatorname{Geom}(G) \to \operatorname{Geom}(H)$ obtained by $(X, \alpha) \to (X, \alpha_{\mid H})$. If $K < G$ is a compact normal subgroup, and $\pi : G \to G/K$ is the associated epimorphism, then there is a natural map $\operatorname{Geom}(G/K) \to \operatorname{Geom}(G)$ obtained by $(X, \alpha) \to (X, \alpha \circ \pi)$. We essentially proved the following.

\begin{proposition}\label{prop:coarse-form}
Let $H$ be a compactly generated locally compact group. Assume that for every pair $\lbrace (X, \alpha), (Y, \beta) \rbrace$ in $\operatorname{Geom}(H)$, the orbital map $X \to Y$ is a rough similarity.
Then 
\begin{enumerate}
\item 
If $K$ is a compact normal subgroup of $H$, then for every pair $\lbrace (X, \alpha), (Y, \beta) \rbrace$ in $\operatorname{Geom}(H/K)$, the orbital map $X \to Y$ is a rough similarity.
\item 
If $G$ receives an injective homomorphism with closed and co-compact image from $H$, then for every pair $\lbrace (X, \alpha), (Y, \beta) \rbrace$ in $\operatorname{Geom}(G)$, the orbital map $X \to Y$ is a rough similarity.
\end{enumerate}
\end{proposition}

Note that two-ended groups have the property in the proposition.
{\color{black}
Also, if $\Gamma=H$ is a finitely generated group which sits as a uniform lattice in a locally compact group $G$, the property expressed by Proposition~\ref{prop:coarse-form} transfers from $\Gamma$ to $G$, but not from $G$ to $\Gamma$.}

\subsection{Final questions}
An affirmative answer to the next question would provide a robust generalization of the main results of the present paper.

\begin{question}
Let $G$ be a completely\footnote{A group is completely solvable if it is isomorphic to a closed subgroup of upper triangular real matrices.} solvable Lie group with $H^1(G, \mathbb R)= \mathbb R$. Does it hold that for every pair $\lbrace (X, \alpha), (Y, \beta) \rbrace$ in $\operatorname{Geom}(G)$, the orbital map $X \to Y$ is a rough similarity?
\end{question}

\begin{question}
Same question as above, where $G = (\mathbb Z_m \times \mathbb Z_m) \rtimes \mathbb Z$ is the locally compact group that contains $L_m$ as a lattice.
\end{question}


\begin{question}
\label{ques:dymarz-th}
Same question as above, where $G=(N \times \mathbb Q_m) \rtimes \mathbb Z$, where $N$ is a nilpotent connected Lie group and $\mathbb Z$ acts by multiplication by $m$ on $\mathbb Q_m$ and by a contracting automorphism on $N$.
\end{question}

If the answer to Question~\ref{ques:dymarz-th} is yes, then the main theorem of \cite{DymarzbdriesAmenHyp} would enter the framework of Theorem~\ref{th:reformulation}.

\bibliographystyle{amsalpha}

\bibliographystyle{alpha}

\bibliography{biblio_left_invariant_metrics}

\end{document}